\documentclass[12pt,a4paper]{article}

\usepackage{geometry}
\usepackage{titling}
\usepackage{amsopn, amstext, amsmath,amsthm, amsbsy,mathrsfs}
\usepackage{latexsym,amssymb}
\usepackage{ragged2e}
\usepackage{url}
\usepackage{bbm}
\usepackage{authblk}
\usepackage{amsfonts}
\usepackage{dsfont}
\usepackage{color}
\usepackage{booktabs}
\usepackage{multirow}
\usepackage{array}

\usepackage{diagbox, fancyvrb}
\usepackage[english]{babel}

\usepackage{setspace} 

\usepackage{titlesec} 
\usepackage{titling}
\titleformat*{\section}{\normalsize\bfseries}
\titleformat*{\subsection}{\normalsize\it\bfseries}
\titlespacing*{\section}{0pt}{6pt}{6pt}

\titlelabel{\thetitle.\quad}

\RequirePackage{bm}

\numberwithin{equation}{section}
\usepackage[square,numbers,sort&compress]{natbib}

\usepackage[colorlinks,
linkcolor=blue,
anchorcolor=blue,
citecolor=blue
]{hyperref}

\newcommand{\sref}[2]{\hyperref[#2]{#1 \ref{#2}}}
\newcommand{\myeqref}[1]{\hyperref[#1]{\eqref{#1} }}

\newcommand\keywords[1]{\text{keywords}: #1 }
\newcommand{\as}{$\operatorname{\mathbb{P}-a.s.}$}
\newtheorem{theorem}{Theorem}[section]  
    
\newtheorem{lemma}{Lemma}[section]  
  
\newtheorem{example}{Example}[section]

\newtheorem{hypothesis}{Hypothesis}[section]
\newtheorem{scheme}{Scheme}[section]
\date{\empty}
\renewenvironment{abstract}{\par\noindent\textbf{\abstractname}\ \ignorespaces}{\par\medskip}

\begin{document}

	\title{\large New Second-order Convergent
		Schemes for Solving decoupled FBSDEs}
	\author[a]{Wenbo Wang} 
	\author[a,b,$ \ast $]{Guangyan Jia}
	\affil[a]{\normalsize Zhongtai Securities Institute for Financial Studies,  Shandong University, Jinan, P.R. China}
	\affil[b]{\normalsize Shandong Province Key Laboratory of Financial Risk, Jinan, P.R. China}
	
	\vspace{0.8cm}
	\maketitle 		\thispagestyle{empty}
	\begin{spacing}{1.25}
		\begin{flushleft}
			{\small	
				\begin{abstract}
					
					\justifying \noindent
					This paper proposes a new second-order symmetric algorithm for solving decoupled forward-backward stochastic differential equations. Inspired by the alternating direction implicit splitting method for partial differential equations, we split the generator into the sum of two functions. In the computation of the value process Y, explicit and implicit schemes are alternately applied to these two generators, while the algorithms from \citep{ZhaoLi2014} are used for the control process Z. We rigorously prove that the two new schemes have second-order convergence rate. The proposed splitting methods show clear advantages for equations whose generator consists of a linear part plus a nonlinear part, as they reduce the number of iterations required for solving implicit schemes, thereby decreasing computational cost while maintaining second-order convergence. Two numerical examples are provided, including the backward stochastic Riccati equation arising in mean-variance hedging. The numerical results verify the theoretical error analysis and demonstrate the advantage of reduced computational cost compared to the algorithm in \citep{ZhaoLi2014}.

	\end{abstract}	
	\keywords{BSDEs, symmetric algorithm, alternating direction implicit, second-order convergent} 
}
\end{flushleft}
\end{spacing}

\section{Introduction}
This paper focuses on  numerical solution of the  forward-backward stochastic differential equation (FBSDE)

\begin{align}\label{eq1}
\begin{cases}
X_t = X_0 + \int_0^t b(s, X_s) \, ds + \int_0^t \sigma(s, X_s) \, dW_s, \quad \text{(SDE)} \\
Y_t = \Phi(X_T) + \int_t^T f(s, X_s, Y_s, Z_s) \, ds - \int_t^T Z_s \, dW_s, \quad \text{(BSDE)}
\end{cases}
\end{align}
where $T>0,\,W = (W^{1},\,\dots,\,W^{d})$ denotes a standard $ d $-dimensional Brownian motion defined on a filtered probability space
$(\Omega,\, \mathcal{F},\, \mathbb{F},\, \mathbb{P})$. Here, $\mathbb{F}=(\mathcal{F}_{s})_{s \in [0, \, T]}$ is the $\mathbb{P}$-completion of the filtration generated by $W$. $X_0$ is the initial condition of the forward stochastic differential equation (SDE). $b:[0, T] \times  \mathbb{R}^{d}  \mapsto \mathbb{R}^{d}$ and $\sigma: [0, T] \times  \mathbb{R}^{d} \mapsto \mathbb{R}^{d\times d}$ are the drift coefficient and the diffusion matrix of the SDE, respectively. $\Phi:\mathbb{R}^{d} \mapsto \mathbb{R}^{k} $ and $f: [0, T] \times  \mathbb{R}^{d} \times \mathbb{R}^{k} \times \mathbb{R}^{k\times d} \mapsto \mathbb{R}^{k} $ are the terminal and generator functions of the backward stochastic differential equation (BSDE), respectively. Note that the two integrals with respect to $W_s$ in \myeqref{eq1} are the It\^o-type integrals.

A solution of FBSDE\myeqref{eq1} is defined as a triple $(X,\,Y,\,Z):=\{(X_{t},\,Y_{t},\,Z_{t})\}_{t \in [0,T]} $ fulfilling that it is \(\mathcal{F}_{t}-\)adapted, square integrable, and \as, \myeqref{eq1} holds. When $ k = 1 $, nonlinear BSDEs are pioneered under the Lipschitz condition by \citep{Peng1990}.  
\citep{Peng1991} proved the nonlinear Feynman-Kac formula and established a profound link between the solutions of FBSDEs and quasilinear parabolic partial differential equations (PDEs). Namely, under specific regularity conditions, the solution \((X,\,Y, Z)\) to FBSDE\myeqref{eq1} admits the following representation
\[ Y_{t} = u(t,X_{t}),\, Z_{t} = (u_{x}\sigma)(t,X_{t}),    \]
where \(u: [0, T] \times  \mathbb{R}^{d} \mapsto \mathbb{R}^{k}\) is the solution to the parabolic PDE below:
\begin{align}\label{eq2}
\partial_t u + \frac{1}{2} \sum_{i,j=1}^d [\sigma\sigma^*]_{i,j} \partial_{x_i x_j}^2 u + \sum_{i=1}^d b_i\partial_{x_i} u + H(t, x, u, u_x \sigma) = 0
\end{align}
with terminal condition $u(T,\, \cdot)=\Phi(\cdot),$ where
\[
u_x(t, x) = 
\begin{bmatrix}
\frac{\partial u_1}{\partial x_1} & \frac{\partial u_1}{\partial x_2} & \cdots & \frac{\partial u_1}{\partial x_d} \\
\frac{\partial u_2}{\partial x_1} & \frac{\partial u_2}{\partial x_2} & \cdots & \frac{\partial u_2}{\partial x_d} \\
\vdots & \vdots & \ddots & \vdots \\
\frac{\partial u_n}{\partial x_1} & \frac{\partial u_n}{\partial x_2} & \cdots & \frac{\partial u_n}{\partial x_d}
\end{bmatrix},
\]
\[
\partial_{x_i} u(t, x) = 
\begin{bmatrix}
\frac{\partial u_1}{\partial x_i} \\
\frac{\partial u_2}{\partial x_i} \\
\vdots \\
\frac{\partial u_n}{\partial x_i}
\end{bmatrix},
\]

\[
\partial_{x_i x_j}^2 u(t, x) = 
\begin{bmatrix}
\frac{\partial^2 u_1}{\partial x_i \partial x_j} \\
\frac{\partial^2 u_2}{\partial x_i \partial x_j} \\
\vdots \\
\frac{\partial^2 u_n}{\partial x_i \partial x_j}
\end{bmatrix}
\]
The nonlinear Feynman-Kac formula not only provides a probabilistic representation for the solutions of a class of partial differential equations, but also characterizes the structure of solutions to FBSDEs. This formula plays a crucial role in the fundamental theoretical research of FBSDEs and also serves as an important research tool for investigating the numerical solutions of FBSDEs. Based on the structure of FBSDE solutions derived from the nonlinear Feynman-Kac formula, this paper studies and proposes fully discrete numerical schemes for the BSDE \myeqref{eq1}.

FBSDEs have found significant applications in numerous research fields, such as mathematical finance\citep{Karoui1997, MaYong1999}, stochastic optimal control\citep{Peng1999}, nonlinear expectation\citep{1997Backward,Gianin2006}, and the theory of partial differential equations (PDEs). Obtaining analytical solutions for FBSDEs is typically a challenging endeavor, thereby driving significant demand for effective numerical methodologies. Recent academic research has witnessed extensive exploration into the development of robust numerical approaches for solving these equations. One of the most commonly employed strategies involves the direct discretization of FBSDEs, giving rise to a wide spectrum of discretization schemes \citep{Ma1994,Milstein2006,Zhao2006,Bender2008,Henry2014}. Popular temporal discretization methods include Euler-type schemes \citep{Gobet2005,Zhang2004}, generalized \(\theta\)-methods \citep{Zhao2012,ZhaoLi2014,Wang2022}, multistep methods \citep{Zhao2010,Zhao2014,Zhao2016,Liu2020,Liu2022}, Runge-Kutta techniques \citep{Chassagneux2014}, , strong stability-preserving multistep methods \citep{Fang2023,Fang2023b}, and extrapolation schemes\citep{XuZhao2024a,XuZhao2024b}, among others. Furthermore, numerous numerical schemes for FBSDEs have been proposed based on the nonlinear Feynman-Kac formula, along with the numerical solutions of parabolic PDEs linked to FBSDEs, as illustrated in the works \citep{ Ma1994,Douglas1996,Milstein2006}.

Recently, \citep{Zheng2025} proposed a splitting method for BSDEs by incorporating the idea of locally one-dimensional (LOD) methods \citep{Yanenko1963} that developed for solving multi-dimensional PDEs. They decomposed the original $d$-dimensional BSDE into $d$ BSDEs and solved them separately—this approach significantly enhances computational efficiency, with the first-order convergence theoretically proven. Inspired by the alternating direction implicit (ADI) method of Peaceman and Rachford \citep{Peaceman1955}, we propose two new second-order convergent schemes for solving decoupled FBSDEs. Complementary to this, \citep{ZhaoLi2014} extended the second-order scheme proposed for BSDEs in \citep{Zhao2006} to the decoupled FBSDEs. By Malliavin calculus theory they proposed two new numerical schemes and proved that the both schemes have second-order convergence rate. In this paper, we split the generator $f$ into the sum of two functions. In the computation of the value process $Y$, explicit and implicit schemes are alternately applied to these two generators, while the algorithms from \citep{ZhaoLi2014} are used for the control process $Z$. We rigorously prove that the two new schemes have second-order convergence rate. The proposed splitting methods show clear advantages for equations whose generator consists of a linear part and a nonlinear part, as they reduce the number of iterations required for solving implicit schemes, thereby decreasing computational cost while maintaining second-order convergence. Two numerical examples are provided, including the backward stochastic Riccati equation arising in mean-variance hedging. The numerical results verify the theoretical error analysis and demonstrate the advantage of reduced computational cost compared to the algorithm in \citep{ZhaoLi2014}.

The structure of this paper is as follows. \sref{Sections}{sec2} introduces preliminaries and splitting schemes. \sref{Sections}{sec3} proves the proposed schemes has second-order convergence in time for solving BSDEs. In\sref{Sections}{sec4}, we  present several numerical tests to show the accuracy of our splitting schemes compared with results in \citep{Zheng2025}.We introduce the following notations:
\begin{itemize}
\item $(\cdot)^\top$: The transpose operator for a matrix or vector; $(\cdot)^{-1}$: the inverse operator for a matrix.
\item $|\cdot|$: The norm for vector or matrix defined by $|c|^2 = \text{trace}(c^\top c)$.
\item $C_b^{l,k,k,k}$: The set of continuously differentiable functions $\psi: [0,T] \times \mathbb R^d \times \mathbb R^n \times \mathbb R^{n \times d} \rightarrow \mathbb R$ with uniformly bounded partial derivatives $\partial_t^{l_1} \psi$ and $\partial_X^{k_1} \partial_Y^{k_2} \partial_Z^{k_3} \psi$ for $l_1 \leq l$ and $k_1 + k_2 + k_3 \leq k$. The notations $C_b^{l,k}$ and $C_b^{l,k,k}$ are similarly defined.
\item $C_b^{k,\alpha}$ ($\alpha \in (0,1)$): $\phi \in C_b^{k,\alpha}$ means that it has uniformly bounded partial derivatives up to order $k$, and $\partial_x^k \phi(x)$ is H\"older continuous with index $\alpha$ w.r.t. $x$.
\item $\mathcal F_s^{t,x}$ ($t \leq s \leq T$): $\sigma$-field generated by the diffusion process $\{X_r, t \leq r \leq s, X_t = x\}$ starting from the time-space point $(t,x)$. When $s = T$, we use $\mathcal F^{t,x}$ to denote $\mathcal F_T^{t,x}$.
\item $\mathbb E_s^{t,x}[\eta]$: The conditional mathematical expectation of the random variable $\eta$ under the $\sigma$-field $\mathcal F_s^{t,x}$ $(t \leq s \leq T)$, i.e., $\mathbb E_s^{t,x}[\eta] = \mathbb E[\eta \mid \mathcal F_s^{t,x}]$. Let $\mathbb E_t^x[\eta] = \mathbb E[\eta \mid \mathcal F_t^{t,x}]$.
\end{itemize}

\section{Preliminaries and splitting schemes}\label{sec2}

Suppose that $H$ is a real separable Hilbert space with scalar product denoted by $\langle \cdot, \cdot \rangle_H$. The norm of an element $h \in H$ will be denoted by $\|h\|_H$. Let $\mathcal{B} = \{B(h), h \in H\}$ denote an isonormal Gaussian process associated with the Hilbert space $H$ on $(\Omega, \mathcal F, \mathbb F, \mathbb P)$.

Let the operator $D^k$ \citep{nualart1997} be the Malliavin derivative of order $k$. For any integer $p \geq 1$, $D^{k,p}$ is the domain of $D^k$ ($k \in \mathbb{N}$) in $L^p(\Omega)$, i.e., $D^{k,p}$ is the closure of the class of smooth random variables $F$ with respect to the norm $\|F\|_{k,p} = \left[ \mathbb E[|F|^p] + \sum_{j=1}^k  \mathbb E[\|D^j F\|_H^p] \right]^{1/p}$. For $p = 2$, the space $D^{1,2}$ is a Hilbert space with the scalar product
\[
\langle F, G \rangle = \mathbb E[FG] + \mathbb E[\langle DF, DG \rangle_H].
\]

For the Malliavin derivative operator $D$, we introduce the following two lemmas.

\begin{lemma}[See \citep{nualart1997}]
For $F \in D^{1,2}$ and $u \in L^2(\Omega; H)$, we have
\[
D_t \int_0^T u_s  ds = \int_t^T D_t u_s  ds, \quad 0 \leq t \leq T, \quad D_t \int_0^T u_s  dW_s = u_t + \int_t^T D_t u_s  dW_s, \quad 0 < t \leq T,
\]
and the following integration-by-parts formula:
\[
\mathbb E\left[ \int_0^T u_t D_t F  dt \right] = \mathbb E\left[ F \int_0^T u_t  dW_t \right].
\]
\end{lemma}

\begin{lemma}[See \citep{nualart1997}]
If the solution $(X_t, Y_t, Z_t)$ to (1.1) is in $D^{1,2}$, we have
\begin{align*}
D_s X_t I_{s \leq t} &= \nabla_x X_t (\nabla_x X_s)^{-1} \sigma(s, X_s) I_{s \leq t}, \\
Z_t &= D_t Y_t = \nabla_x Y_t (\nabla_x X_t)^{-1} \sigma(t, X_t), \\
D_s Y_t I_{s \leq t} &= \nabla_x Y_t (\nabla_x X_s)^{-1} \sigma(s, X_s) I_{s \leq t} = Z_t \sigma^{-1}(t, X_t) \nabla_x X_t (\nabla_x X_s)^{-1} \sigma(s, X_s) I_{s \leq t}.
\end{align*}
\end{lemma}

For $0\le t \le s\le T, $ let $ \Delta W_{t,s}:= (\Delta W^{1}_{t,s},\dots,\Delta W^{d}_{t,s})^{T}$ represent the increment $W_{s}-W_{t}$ of the Brownian motion $(W_{t})_{t\in[0,T]}.$ 
For the time interval $[0,\,T],$  we consider the uniform partition \( t_{n} = n \Delta t,\,n = 0,\,1,\,\dots, N\) with $\Delta t  = T/N.$ Let $\Delta W_{t_{n+1}}:= \Delta W_{t_{n},t_{n+1}}.$ Then $\Delta W_{t_{n+1}}$ admits the Gaussian distribution with mean zero and variance $\Delta t.$

\subsection{Numerical schemes for solving SDEs}
Since the FBSDEs (1.1) are decoupled, the SDE and the BSDE in (1.1) can be solved separately. There have been many numerical schemes for solving SDEs in literature, such as the Euler scheme, the Milstein scheme and the It\^o-Taylor type schemes \citep{Kloeden1992}.

We assume that the numerical scheme used for solving the SDE in (1.1) is in the following form:
\begin{align}\label{eq3}
X^{n+1} = X^n + \psi(t_n, X^n, \Delta t, \xi_n^{n+1}), \quad n = 0,1,\ldots,N-1,
\end{align}
where $X^0 = X_0$, $\phi$ is a given $\mathbb R^d$-valued function, and $\xi_n^{n+1}$ is a random vector related to $\Delta W_{t_{n+1}} = W_{t_{n+1}} - W_{t_n}$, , and under some conditions on the coefficients $b$ and $\sigma$ the function $\psi$ satisfiesthe following two Hypothesis.

\begin{hypothesis}\label{H2.1}  For $n = 0, 1, \ldots, N-1$, the function $\phi$ in equation (2.1) has the estimate
\[
\mathbb E_{t_n}^{X^n}\left[|\psi(t_n, X^n, \Delta t, \xi_n^{n+1})|^2\right] \leq C\Delta t,
\]
and the Malliavin derivative $D_t X^{n+1}$ of $X^{n+1}$ defined by \myeqref{eq3} has the approximate property
\[
\mathbb E_{t_n}^{X^n}\left[|D_t X^{n+1} - \sigma^n|\right] \leq C\sqrt{\Delta t}, \quad t_n < t \leq t_{n+1}.
\]
\end{hypothesis}
\begin{hypothesis} \label{H2.2} The approximation solution $X^{n+1}$ to \myeqref{eq3}, for $n = 0, 1, 2, \ldots, N-1$, have the properties
\begin{align*}
&\left|\mathbb E_{t_n}^{X^n}\left[g(X_{t_{n+1}}^{t_n, X^n}) - g(X^{n+1})\right]\right| \leq C_g(\Delta t)^{\beta+1}, \\
&\left|\mathbb E_{t_n}^{X^n}\left[(g(X_{t_{n+1}}^{t_n, X^n}) - g(X^{n+1}))\Delta W_{t_{n+1}}^\top\right]\right| \leq C_g(\Delta t)^{\gamma+1},
\end{align*}
where $g \in C_b^{2\beta+2}$, $\beta$ and $\gamma$ are nonnegative numbers, and $C_g$ is a positive number which does not depend on the time partition.
\end{hypothesis}
Numerical schemes for one-dimensional and multi-dimensional SDEs have been extensively studied (see e.g.\citep{Kloeden1992}), and many of them can be written in the form of equation (2.1), where the function $\phi$ satisfies the above two hypotheses. The following are three examples for one-dimensional SDEs that illustrate Hypothesis 2.1:

\begin{enumerate}
\item \textbf{The Euler scheme} is in the form of (2.1) with $\phi$ defined by
\[
\psi = b^n\Delta t + \sigma^n\xi_n^{n+1}, \quad \xi_n^{n+1} = \Delta W_{t_{n+1}}.
\]
The Malliavin derivative is given by
\begin{equation*}
D_t X^{n+1} = \sigma^n, \quad t_n < t \leq t_{n+1}. 
\end{equation*}

\item \textbf{The Milstein scheme} is in the form of (2.1) with $\phi$ defined by
\[
\psi = b^n\Delta t + \sigma^n\xi_n^{n+1} + \frac{1}{2}\sigma^n\sigma_x^n\left((\xi_n^{n+1})^2 - \Delta t\right), \quad \xi_n^{n+1} = \Delta W_{t_{n+1}}.
\]
The Malliavin derivative is given by
\begin{equation*}
D_t X^{n+1} = \sigma^n + \sigma^n\sigma_x^n\Delta W_{t_{n+1}}, \quad t_n < t \leq t_{n+1}. 
\end{equation*}

\item \textbf{The order-2 weak It\^o-Taylor type scheme} is given by
\begin{align*}
X^{n+1} = & X^n + b^n\Delta t + \sigma^n\xi_n^{n+1} + \frac{1}{2}\sigma^n\sigma_x^n\left((\xi_n^{n+1})^2 - \Delta t\right) \\
& + \frac{1}{2}\left[\sigma_t^n + \sigma^n b_x^n + b^n\sigma_x^n + \frac{1}{2}(\sigma^n)^2\sigma_{xx}^n\right]\Delta t\xi_n^{n+1} \\
& + \frac{1}{2}\left[b_t^n + b^n b_x^n + (\sigma^n)^2 b_{xx}^n\right](\Delta t)^2, 
\end{align*}
with $\xi_n^{n+1} = \Delta W_{t_{n+1}}$.

The Malliavin derivative is given by
\begin{align*}
D_t X^{n+1} = & \sigma^n + \sigma^n\sigma_x^n\Delta W_{t_{n+1}} \\
& + \frac{\Delta t}{2}\left[\sigma_t^n + \sigma^n b_x^n + b^n\sigma_x^n + \frac{1}{2}(\sigma^n)^2\sigma_{xx}^n\right], \quad t_n < t \leq t_{n+1}. 
\end{align*}
\end{enumerate}

Here, for a function $v = v(t, x)$, we use $v^n$ to denote $v(t_n, X^n)$.

\subsection{Reference equations for solving BSDEs}
Let the triple $(X_t^{t_n, X^n}, Y_t^{t_n, X^n}, Z_t^{t_n, X^n})_{t \in [t_n, T]}$ be the solution to the following FBSDE:
\begin{align}\label{eq2.2}
\begin{cases}
X_t^{t_n, X^n} = X^n + \int_{t_n}^t b(s, X_s^{t_n, X^n})  ds + \int_{t_n}^t \sigma(s, X_s^{t_n, X^n})  dW_s, \\
Y_t^{t_n, X^n} = \Phi(X_T^{t_n, X^n}) + \int_t^T f(s, X_s^{t_n, X^n}, Y_s^{t_n, X^n}, Z_s^{t_n, X^n})  ds - \int_t^T Z_s^{t_n, X^n}  dW_s
\end{cases}
\end{align}
Then
\begin{align}\label{eq2.3}
Y_{t_n}^{t_n, X^n} = Y_{t_{n+1}}^{t_n, X^n} + \int_{t_n}^{t_{n+1}} f_s^{t_n, X^n} ds - \int_{t_n}^{t_{n+1}} Z_s^{t_n, X^n} dW_s,
\end{align}
where $f_s^{t_n, X^n} := f(s, X_s^{t_n, X^n}, Y_s^{t_n, X^n}, Z_s^{t_n, X^n})$. Taking $\mathbb E_{t_n}^{X^n}[\cdot]$ on both sides of \myeqref{eq2.3} leads to
\begin{align}
Y_{t_n}^{t_n, X^n} = \mathbb E_{t_n}^{X^n}[Y_{t_{n+1}}^{t_n, X^n}] + \int_{t_n}^{t_{n+1}} \mathbb E_{t_n}^{X^n}[f_s^{t_n, X^n}] ds.
\end{align}
Next,\citep{ZhaoLi2014} used trapezoidal rule to approximate $\int_{t_n}^{t_{n+1}} \mathbb E_{t_n}^{X^n}[f_s^{t_n, X^n}] ds,$ while we use an alternating implicit scheme.Specifically, we define $f_{i}:[0, T] \times  \mathbb{R}^{d} \times  \mathbb{R}^{k} \times  \mathbb{R}^{k\times d}\mapsto \mathbb{R}^{k},\,i=1,2 $ such that $f=f_{1}+f_{2}.$  Without loss of generality, we assume $N$ is an even number throughout this paper. For $k = 0,1,\dots,N/2-1,$ consider the following cases.

Case 1: If $n=2k,$ we have
\begin{align}\label{evenyn}
Y_{t_n}^{t_n, X^n} &= \mathbb E_{t_n}^{X^n}[Y_{t_{n+1}}^{t_n, X^n}] + \Delta t f_{1,t_n}^{t_n, X^n} +  \Delta t \mathbb E_{t_n}^{X^n}[f_{2,t_{n+1}}^{t_n, X^n}] + R_Y^{n} \nonumber\\
&= \mathbb E_{t_n}^{X^n}[Y_{t_{n+1}}^{t_{n+1}, X^{n+1}}] + \Delta t f_{1,t_n}^{t_n, X^n} +\Delta t \mathbb E_{t_n}^{X^n}[f_{2,t_{n+1}}^{t_{n+1}, X^{n+1}}] + R_Y^{n} + R_{Y1}^{n},
\end{align}
where
\begin{equation}\label{Rny}
R_Y^{n} = \int_{t_n}^{t_{n+1}} \left\{ \mathbb E_{t_n}^{X^n}[f_s^{t_n, X^n}] - \mathbb E_{t_n}^{X^n}[f_{2,t_{n+1}}^{t_n, X^n}] - f_{1,t_n}^{t_n, X^n} \right\} ds, 
\end{equation}
\begin{equation}\label{Rny1}
R_{Y1}^{n} = \mathbb E_{t_n}^{X^n}[Y_{t_{n+1}}^{t_n, X^n} - Y_{t_{n+1}}^{t_{n+1}, X^{n+1}}] + \Delta t \mathbb E_{t_n}^{X^n}[f_{2,t_{n+1}}^{t_n, X^n} - f_{2,t_{n+1}}^{t_{n+1}, X^{n+1}}].
\end{equation}
Case 2: If $n=2k+1$, we have
\begin{align}\label{oddyn}
Y_{t_n}^{t_n, X^n} &= \mathbb E_{t_n}^{X^n}[Y_{t_{n+1}}^{t_n, X^n}] +  \Delta t f_{2,t_n}^{t_n, X^n} +  \Delta t \mathbb E_{t_n}^{X^n}[f_{1,t_{n+1}}^{t_n, X^n}] +\tilde R_Y^{n}\nonumber \\
&= \mathbb E_{t_n}^{X^n}[Y_{t_{n+1}}^{t_{n+1}, X^{n+1}}] + \Delta t f_{2,t_n}^{t_n, X^n} + \Delta t \mathbb E_{t_n}^{X^n}[f_{1,t_{n+1}}^{t_{n+1}, X^{n+1}}] + \tilde R_Y^{n} + \tilde R_{Y1}^{n},
\end{align}
where
\begin{equation}\label{tRny}
\tilde R_Y^{n} = \int_{t_n}^{t_{n+1}} \left\{ \mathbb E_{t_n}^{X^n}[f_s^{t_n, X^n}] -  \mathbb E_{t_n}^{X^n}[f_{1,t_{n+1}}^{t_n, X^n}] -  f_{2,t_n}^{t_n, X^n} \right\} ds, 
\end{equation}
\begin{equation}\label{tRny1}
\tilde R_{Y1}^{n} = \mathbb E_{t_n}^{X^n}[Y_{t_{n+1}}^{t_n, X^n} - Y_{t_{n+1}}^{t_{n+1}, X^{n+1}}] + \Delta t \mathbb E_{t_n}^{X^n}[f_{1,t_{n+1}}^{t_n, X^n} - f_{1,t_{n+1}}^{t_{n+1}, X^{n+1}}].
\end{equation}

Note that $R_n, R_{Y1}^{n} \tilde R_Y^{n},$ and $\tilde R_{Y1}^{n}$ are defined for all $n = 0,1,\dots,N,$ not only for even or odd $ n$. This is useful for the subsequent error estimates. We use two schemes for solving $Z$ in \citep{ZhaoLi2014} (see equation (2.19) and (2.24) in \citep{ZhaoLi2014} ), they are respectively
\begin{align} \label{zn1}
\Delta t_n Z_{t_n}^{t_n, X^n} = & -\frac{1}{2}\Delta t \mathbb{E}_{t_n}^{X^n}\left[Z_{t_{n+1}}^{t_{n+1}, X^{n+1}} - Z_{t_{n+1}}^{t_{n+1}, X^{n+1}}[\sigma^{n+1}]^{-1} A^n\sigma^n\right]  + \mathbb{E}_{t_n}^{X^n}\left[Y_{t_{n+1}}^{t_{n+1}, X^{n+1}}\Delta W_{t_{n+1}}^\top\right] \nonumber\\
& + \Delta t \mathbb{E}_{t_n}^{X^n}\left[f(t_{n+1}, X^{n+1}, Y^{n+1}, Z^{n+1})\Delta W_{t_{n+1}}^\top\right]  + 2 R_Z^n +  R_{Z 1}^n,
\end{align}

\begin{align} \label{zn2}
\Delta t Z_{t_n}^{t_n, X^n} = & \mathbb{E}_{t_n}^{X^n}\left[Z_{t_{n+1}}^{t_{n+1}, X^{n+1}}\left([\sigma^{n+1}]^{-1}\int_{t_n}^{t_{n+1}} D_t X^{n+1} dt - \Delta t\right)\right]+ \mathbb{E}_{t_n}^{X^n}\left[Y_{t_{n+1}}^{t_{n+1}, X^{n+1}}\Delta W_{t_{n+1}}^\top\right] \nonumber\\
& + \Delta t \mathbb{E}_{t_n}^{X^n}\left[f(t_{n+1}, X^{n+1}, Y^{n+1}, Z^{n+1})\Delta W_{t_{n+1}}^\top\right] + 2 R_Z^n + R_{Z 2}^n,
\end{align}
where $\sigma^{n} = \sigma(t_{n}, X^{n})$, 
\(A^{n} = I_{d} + \partial_{x}b^{n}\Delta t + \sum_{j=1}^{d}\partial_{x}\sigma_{j}^{n}\Delta W_{t_{n+1}}^{j},I_{d}\) is the $d\times d$ identity matrix, $\sigma_{j}(\cdot)$ is the $j$-th column of the matrix $\sigma(\cdot)$,
\([\sigma^{n+1}]^{-1}\) is the inverse matrix of $\sigma^{n+1},$ and
\begin{align} \label{Rnz}
R_Z^n = & \int_{t_n}^{t_{n+1}} \left\{ \mathbb{E}_{t_n}^{X^n}[f_s^{t_n, X^n}\Delta W_s^\top] - \frac{1}{2}\mathbb{E}_{t_n}^{X^n}[f_{t_{n+1}}^{t_{n+1}, X^{n+1}}\Delta W_{t_{n+1}}^\top] \right\} ds \nonumber\\
& - \int_{t_n}^{t_{n+1}} \left\{ \mathbb{E}_{t_n}^{X^n}[Z_s^{t_n, X^n}] - \frac{1}{2}Z_{t_n}^{t_n, X^n} - \frac{1}{2}\mathbb{E}_{t_n}^{X^n}[Z_{t_{n+1}}^{t_n, X^n}] \right\} ds,
\end{align}

\begin{align}\label{Rnz1}
R_{Z 1}^n &= \int_{t_n}^{t_{n+1}} \mathbb{E}_{t_n}^{X^n}\left[D_t Y_{t_{n+1}}^{t_n, X^n} - \frac{1}{2}D_{t_n} Y_{t_{n+1}}^{t_n, X^n} - \frac{1}{2}D_{t_{n+1}} Y_{t_{n+1}}^{t_n, X^n}\right] dt
\nonumber \\
& +  \frac{1}{2}\Delta t \mathbb{E}_{t_n}^{X^n}\left[Z_{t_{n+1}}^{t_n, X^n}\sigma^{-1}(t_{n+1}, X_{t_{n+1}}^{t_n, X^n})(\nabla_x X_{t_{n+1}}^{t_n, X^n} - A^n)\right]\sigma(t_n, X^n)
\nonumber \\
& +  \mathbb{E}_{t_n}^{X^n}\left[(Y_{t_{n+1}}^{t_n, X^n} - Y_{t_{n+1}}^{t_{n+1}, X^{n+1}} )\Delta W_{t_{n+1}}^\top\right]- \frac{1}{2}\Delta t \mathbb{E}_{t_n}^{X^n}\left[Z_{t_{n+1}}^{t_n, X^n} - Z_{t_{n+1}}^{t_{n+1}, X^{n+1}}\right] \nonumber\\
&  + \frac{1}{2}\Delta t \mathbb{E}_{t_n}^{X^n}\left[ (Z_{t_{n+1}}^{t_n, X^n}\sigma^{-1}(t_{n+1}, X_{t_{n+1}}^{t_n, X^n})- Z_{t_{n+1}}^{t_{n+1}, X^{n+1}}\sigma^{-1}(t_{n+1}, X^{n+1})) A^n\sigma(t_n, X^n)\right],
\end{align}
\begin{align}\label{Rnz2}
R_{Z2}^n = &-\Delta t \mathbb{E}_{t_n}^{X^n}\left[Z_{t_{n+1}}^{t_n, X^n} - Z_{t_{n+1}}^{t_{n+1}, X^{n+1}}\right]+ 2\mathbb{E}_{t_n}^{X^n}\left[\left(Y_{t_{n+1}}^{t_n, X^n} - Y_{t_{n+1}}^{t_{n+1}, X^{n+1}}\right)\Delta W_{t_{n+1}}^\top\right].
\end{align}

In the following subsections, we will propose our first scheme based on the three reference equations \myeqref{evenyn}, \myeqref{oddyn}, and \myeqref{zn1}. Our second scheme based on reference equations \myeqref{evenyn}, \myeqref{oddyn}, and \myeqref{zn2}.
\subsection{Schemes}
We are ready to present our splitting schemes for BSDEs.
\begin{scheme}\label{scheme1}
Given $Y^{N} = \Phi(X^{N})$ and $Z^{N} = \Phi_{x}(X^{N})\sigma(t_{N}, X^{N})$, for $n = N-1, \ldots, 1, 0,\,k = 0,1,\dots,N/2-1$, solve random variables $X^{n+1}$, $Z^{n}$ and $Y^{n}$ by
\begin{align*}
X^{n+1} &= X^{n} + \psi(t_{n}, X^{n}, \Delta t, \xi_{n}^{n+1}), \\
\Delta tZ^{n} &= -\frac{1}{2}\Delta t\mathbb E_{t_{n}}^{X^{n}}\left[Z^{n+1} - Z^{n+1}[\sigma^{n+1}]^{-1}A^{n}\sigma^{n}\right] + \mathbb E_{t_{n}}^{X^{n}}\left[Y^{n+1}\Delta W_{t_{n+1}}^{\top}\right] \\
&\quad + \Delta t\mathbb E_{t_{n}}^{X^{n}}\left[f(t_{n+1}, X^{n+1}, Y^{n+1}, Z^{n+1})\Delta W_{t_{n+1}}^{\top}\right], \\
Y^{n} &= \mathbb E_{t_{n}}^{X^{n}}\left[Y^{n+1}\right] + \Delta tf_{1}(t_{n}, X^{n}, Y^{n}, Z^{n}) + \Delta t\mathbb E_{t_{n}}^{X^{n}}\left[f_{2}(t_{n+1}, X^{n+1}, Y^{n+1}, Z^{n+1})\right],\text{if } n=2k,\\
Y^{n} &= \mathbb E_{t_{n}}^{X^{n}}\left[Y^{n+1}\right] + \Delta tf_{2}(t_{n}, X^{n}, Y^{n}, Z^{n}) + \Delta t\mathbb E_{t_{n}}^{X^{n}}\left[f_{1}(t_{n+1}, X^{n+1}, Y^{n+1}, Z^{n+1})\right],\text{if } n=2k+1,
\end{align*}
where $\sigma^{n} = \sigma(t_{n}, X^{n})$, 
\(A^{n} = I_{d} + \partial_{x}b^{n}\Delta t + \sum_{j=1}^{d}\partial_{x}\sigma_{j}^{n}\Delta W_{t_{n+1}}^{j},I_{d}\) is the $d\times d$ identity matrix, $\sigma_{j}(\cdot)$ is the $j$-th column of the matrix $\sigma(\cdot)$,
and $[\sigma^{n+1}]^{-1}$ is the inverse matrix of $\sigma^{n+1}$.

\end{scheme}

\begin{scheme}\label{scheme2}
Given $Y^{N} = \Phi(X^{N})$ and $Z^{N} = \Phi_{x}(X^{N})\sigma(t_{N}, X^{N})$, for $n = N-1, \ldots, 1, 0,\,k = 0,1,\dots,N/2-1$, solve random variables $X^{n+1}$, $Z^{n}$ and $Y^{n}$ by
\begin{align*}
X^{n+1} &= X^{n} + \psi(t_{n}, X^{n}, \Delta t, \xi_{n}^{n+1}), \\
\Delta tZ^{n} &= \mathbb E_{t_{n}}^{X^{n}}\left[Z^{n+1}\left([\sigma^{n+1}]^{-1}\int_{t_{n}}^{t_{n+1}}D_{t}X^{n+1} \u d t - \Delta t\right)\right]  + \mathbb E_{t_{n}}^{X^{n}}\left[Y^{n+1}\Delta W_{t_{n+1}}^{\top}\right] \\
&\quad + \Delta t\mathbb E_{t_{n}}^{X^{n}}\left[f(t_{n+1}, X^{n+1}, Y^{n+1}, Z^{n+1})\Delta W_{t_{n+1}}^{\top}\right], \\
Y^{n} &= \mathbb E_{t_{n}}^{X^{n}}\left[Y^{n+1}\right] + \frac{1}{2}\Delta tf_{1}(t_{n}, X^{n}, Y^{n}, Z^{n})+ \frac{1}{2}\Delta t\mathbb E_{t_{n}}^{X^{n}}\left[f_{2}(t_{n+1}, X^{n+1}, Y^{n+1}, Z^{n+1})\right],\text{if } n=2k,\\
Y^{n} &= \mathbb E_{t_{n}}^{X^{n}}\left[Y^{n+1}\right] + \frac{1}{2}\Delta tf_{2}(t_{n}, X^{n}, Y^{n}, Z^{n})+ \frac{1}{2}\Delta t\mathbb E_{t_{n}}^{X^{n}}\left[f_{1}(t_{n+1}, X^{n+1}, Y^{n+1}, Z^{n+1})\right],\text{if } n=2k+1,
\end{align*}
where $\sigma^{n} = \sigma(t_{n}, X^{n})$, and $D_{t}$ is the Malliavin derivative.

\end{scheme}

\section{Error analysis}\label{sec3}
\subsection{Gneneral error estimate}

\begin{hypothesis}\label{H3.1}
The functions $|\partial_x b|$, $\sum_{j=1}^d |\partial_x \sigma_j|$, and $|(\sigma)^{-1}|$ are bounded. The functions $ f_{i}(t, x, y, z),\,i=1,2$ are uniformly Lipschitz continuous with respect to $(x, y, z) \in \mathbb{R}^d \times \mathbb{R}^k \times \mathbb{R}^{k \times d}$.
\end{hypothesis}

\begin{theorem}\label{Theorem3.1}
Let $\left(X_t^{t_n, X^n}, Y_t^{t_n, X^n}, Z_t^{t_n, X^n}\right)_{t_n \leq t \leq T}(n = 0, 1, \ldots, N)$ be the solutions to the FBSDEs \myeqref{eq2.2}  and $\left(X^n, Y^n, Z^n\right)(n = 0, 1, \ldots, N)$  be the solutions to Scheme 2.1(\(l=1\) in \myeqref{eq3.1}) or Scheme 2.2(\(l=2\) in \myeqref{eq3.1}). Under \sref{Hypothesis}{H2.1} and \sref{Hypothesis}{H2.2}, for sufficiently small time step $\Delta t$, we have 
\begin{align}\label{eq3.1}
& \left|Y_{t_n}^{t_n, X^n} - Y^n\right|^2 + \frac{1}{9} \Delta t \left|Z_{t_n}^{t_n, X^n} - Z^n\right|^2 \nonumber\\
& \leq C \mathbb{E}_{t_n}^{X^n} \left[  \left|Y_{t_{N-1}}^{t_{N-1}, X^{N-1}} - Y^{N-1}\right|^2 + \left|Y_{t_N}^{t_N, X^N} - Y^N\right|^2 +\frac{1}{9} \Delta t \left|Z_{t_N}^{t_N, X^N} - Z^N\right|^2 \right] \nonumber\\
&  \frac{2}{ \Delta t}\sum\limits_{i=n}^{N-2}\left\{ \mathbb{E}_{t_n}^{X^n} \left[\lvert R_Y^{i}+\tilde R_Y^{i+1}\rvert^2+\lvert\tilde R_Y^{i}+ R_Y^{i+1}\rvert^2\right]\right\}\nonumber\\	&	+  \frac{2}{\Delta t}\sum\limits_{i=n}^{N-1}\mathbb E_{t_{n}}^{X^{n}}\left[\lvert R_{Y1}^{i}\rvert^2+\lvert\tilde R_{Y1}^{i}\rvert^2+|R_{Z}^{i}|^{2}+|R_{Zl}^{i}|^{2}\right];
\end{align}

where $C>0$ is independent of $n$ and $ \Delta t,$ $R_Y^n$, $R_{Y1}^n$, $\tilde R_Y^n$, $\tilde R_{Y1}^n$,$R_Z^n$, $R_{Z1}^n$, and $R_{Z2}^n$ are defined in equations  \myeqref{Rny}, \myeqref{Rny1},\myeqref{tRny}, \myeqref{tRny1}, \myeqref{Rnz}, \myeqref{Rnz1}, and \myeqref{Rnz2}, respectively.
\end{theorem}

\begin{proof} Let$\left(X^n, Y^n, Z^n\right)$ for $n = 0, 1, \ldots, N$ be the solutions to the Scheme 2.1. Without loss of generality, we assume that $n$ is even. 
Define 
\[
e_y^n = Y_{t_n}^{t_n, X^n} - Y^n, \quad e_z^n = Z_{t_n}^{t_n, X^n} - Z^n, e^n_{f} = f(t_n, X^n, Y_{t_n}^{t_n, X^n}, Z_{t_n}^{t_n, X^n}) - f(t_n, X^n, Y^n, Z^n),\]\[
e^n_{fi} = f_{i}(t_n, X^n, Y_{t_n}^{t_n, X^n}, Z_{t_n}^{t_n, X^n}) - f_{i}(t_n, X^n, Y^n, Z^n),i=1,2.
\]

By \myeqref{eq2.2} and Scheme 2.1,
\begin{equation}\label{eq3.2}
e_y^n = \mathbb{E}_{t_n}^{X^n} \left[e_y^{n+1}\right]+  \Delta t e^n_{f1}+ \Delta t \mathbb E_{t_n}^{X^n}[e_{f2}^{n+1}] + R_Y^{n}+ R_{Y1}^{n}    
\end{equation}
\begin{equation}\label{eq3.3}
e_y^{n+1} = \mathbb{E}_{t_{n+1}}^{X^{n+1}} \left[e_y^{n+2}\right]+  \Delta t e^{n+1}_{f2}+ \Delta t \mathbb E_{t_{n+1}}^{X^{n+1}}[e_{f1}^{n+2}] + \tilde R_Y^{n+1}+ \tilde R_{Y1}^{n+1}
\end{equation}
Substituting \myeqref{eq3.3} into \myeqref{eq3.2}, we have
\begin{equation}\label{eq3.4}
e_y^n = \mathbb{E}_{t_n}^{X^n} \left[e_y^{n+2}\right]+  \Delta t e^n_{f1}+ \Delta t \mathbb E_{t_n}^{X^n}[2e_{f2}^{n+1}+e_{f1}^{n+2}]+ R_Y^{n}+\mathbb{E}_{t_n}^{X^n} \left[\tilde R_Y^{n+1}\right]+ R_{Y1}^{n}+ \mathbb{E}_{t_n}^{X^n} \left[\tilde R_{Y1}^{n+1}\right]
\end{equation}
Taking $|\cdot|$ at the both sides of \myeqref{eq3.4}, By H3.1 we have
\begin{align}\label{eq3.5}
\lvert	e_y^n \rvert& \leq 	\lvert\mathbb{E}_{t_n}^{X^n} \left[e_y^{n+2}\right]\rvert+ L_{1} \Delta t (\lvert	e_y^n \rvert+\lvert	e_z^n \rvert)+2 L_{2}\Delta t \mathbb E_{t_n}^{X^n}\left[\lvert	e_y^{n+1} \rvert+\lvert	e_z^{n+1} \rvert\right] \nonumber\\	&+L_{1}\Delta t \mathbb{E}_{t_n}^{X^n} \left[\lvert	e_y^{n+2} \rvert+\lvert	e_z^{n+2} \rvert\right] 
+\lvert R_Y^{n}+\mathbb{E}_{t_n}^{X^n} \left[\tilde R_Y^{n+1}\right]\rvert+ \lvert R_{Y1}^{n}\rvert+\lvert \mathbb{E}_{t_n}^{X^n} \left[\tilde R_{Y1}^{n+1}\right]\rvert,
\end{align}
where $L_{1},L_{2}$ are Lipschitz constants of $f_{1},f_{2}$,respectively. By $(a+b)^2\leq (1+\gamma \Delta t)a^2+(1+\frac{1}{\gamma \Delta t})b^2,$for any $\gamma>0,$ we have
\begin{align}\label{eq3.6}
\lvert	e_y^n \rvert^2& \leq (1+\gamma \Delta t)	\lvert\mathbb{E}_{t_n}^{X^n} \left[e_y^{n+2}\right]\rvert^{2}\nonumber\\	& +C_1(1+\frac{1}{\gamma \Delta t})(\Delta t)^2 \left\{\lvert	e_y^n \rvert^2+\lvert	e_z^n \rvert^2+ \mathbb E_{t_n}^{X^n}\left[\lvert	e_y^{n+1} \rvert^2+\lvert	e_z^{n+1} \rvert^2+\lvert	e_y^{n+2} \rvert^2+\lvert	e_z^{n+2} \rvert^2\right]\right\} 
\nonumber\\	&+C_1(1+\frac{1}{\gamma \Delta t})\left\{\lvert R_Y^{n}+\mathbb{E}_{t_n}^{X^n} \left[\tilde R_Y^{n+1}\right]\rvert^2+ \lvert R_{Y1}^{n}\rvert^2+\lvert \mathbb{E}_{t_n}^{X^n} \left[\tilde R_{Y1}^{n+1}\right]\rvert^2\right\},
\end{align}
where $C_1>0$ depends only on $L_{1},L_{2}.$ Then we have
\begin{align}\label{eq3.7}
(1-\frac{C_1\Delta t}{\gamma })	\lvert	e_y^n \rvert^2& \leq (1+\gamma \Delta t)	\lvert\mathbb{E}_{t_n}^{X^n} \left[e_y^{n+2}\right]\rvert^{2}\nonumber\\	& +\frac{C_1\Delta t}{\gamma } \left\{\lvert	e_z^n \rvert^2+ \mathbb E_{t_n}^{X^n}\left[\lvert	e_y^{n+1} \rvert^2+\lvert	e_z^{n+1} \rvert^2+\lvert	e_y^{n+2} \rvert^2+\lvert	e_z^{n+2} \rvert^2\right]\right\} \nonumber\\	& +C_1(\Delta t)^2 \left\{\lvert	e_y^n \rvert^2+\lvert	e_z^n \rvert^2+ \mathbb E_{t_n}^{X^n}\left[\lvert	e_y^{n+1} \rvert^2+\lvert	e_z^{n+1} \rvert^2+\lvert	e_y^{n+2} \rvert^2+\lvert	e_z^{n+2} \rvert^2\right]\right\} 
\nonumber\\	&+C_1(1+\frac{1}{\gamma \Delta t})\left\{\lvert R_Y^{n}+\mathbb{E}_{t_n}^{X^n} \left[\tilde R_Y^{n+1}\right]\rvert^2+ \lvert R_{Y1}^{n}\rvert^2+\lvert \mathbb{E}_{t_n}^{X^n} \left[\tilde R_{Y1}^{n+1}\right]\rvert^2\right\}.
\end{align}
By \myeqref{zn1} and Scheme 2.1,
\begin{equation}\label{eq3.8}
\Delta te_z^n = -\frac{1}{2}\Delta t\mathbb{E}_{t_n}^{X^n} \left[e_z^{n+1}(I_{d}-[\sigma^{n+1}]^{-1}A^{n}\sigma^{n})\right]+ \mathbb E_{t_{n}}^{X^{n}}\left[e_{y}^{n+1}\Delta W_{t_{n+1}}^{\top}\right] +   \Delta t \mathbb E_{t_n}^{X^n}[e_{f}^{n+1}\Delta W_{t_{n+1}}^{\top}] + 2 R_Z^n +  R_{Z 1}^n.
\end{equation}
Taking$|\cdot|^2$, by\((a_{1}+\dots+a_{6})^2\le6(a_{1}^2+\dots+a_{6}^2) \), H\"older's inequalityand H3.1, we have
\begin{align}\label{eq3.9}
(\Delta t)^2\lvert	e_z^n \rvert^2	&\leq \frac{3}{2}(\Delta t)^2\lvert\mathbb{E}_{t_n}^{X^n} \left[e_z^{n+1}(I_{d}-[\sigma^{n+1}]^{-1}A^{n}\sigma^{n})\right] \rvert^2 +6\lvert\mathbb E_{t_{n}}^{X^{n}}\left[e_{y}^{n+1}\Delta W_{t_{n+1}}^{\top}\right]\rvert^2\nonumber\\	& +6	(\Delta t)^3L^{2}\mathbb E_{t_n}^{X^n}\left[\lvert	e_y^{n+1} \rvert^2+\lvert	e_z^{n+1} \rvert^2\right] +12\lvert R_Z^n \rvert^2+6\lvert R_{Z1}^n \rvert^2,
\end{align} 
where $L$ is Lipschitz constant of $f$. Since\[\mathbb E_{t_{n}}^{X^{n}}\left[e_{y}^{n+1}\Delta W_{t_{n+1}}^{\top}\right]=\mathbb E_{t_{n}}^{X^{n}}\left[(e_{y}^{n+1}-\mathbb E_{t_{n}}^{X^{n}}\left[e_{y}^{n+1}\right])\Delta W_{t_{n+1}}^{\top}\right], \]
By H\"older's inequality,
\begin{align}\label{eq3.10}
\lvert E_{t_{n}}^{X^{n}}\left[e_{y}^{n+1}\Delta W_{t_{n+1}}^{\top}\right] \rvert^2	\leq \Delta t \mathbb E_{t_{n}}^{X^{n}}\left[(e_{y}^{n+1}-\mathbb E_{t_{n}}^{X^{n}}\left[e_{y}^{n+1}\right])^2\right]\leq \Delta t (\mathbb E_{t_n}^{X^n}\left[\lvert	e_y^{n+1} \rvert^2\right] -\lvert\mathbb E_{t_{n}}^{X^{n}}\left[e_{y}^{n+1}\right]\rvert^2),
\end{align} 
Under the conditions of the theorem, we have
\begin{align*}
\mathbb E_{t_{n}}^{X^{n}}\left[|\sigma^{n+1}-\sigma^{n}|^{2}\right] 
&= \mathbb E_{t_{n}}^{X^{n}}\left[|\sigma(t_{n+1},X^{n+1})-\sigma(t_{n},X^{n})|^{2}\right] \nonumber\\
&\leqslant \mathbb E_{t_{n}}^{X^{n}}\left[L^{2}(\Delta t+|X^{n+1}-X^{n}|)^{2}\right] \nonumber\\
&\leqslant 2L^{2}(\Delta t)^{2}+2L^{2}\mathbb E_{t_{n}}^{X^{n}}\left[|\psi(t_{n},X^{n},\Delta t,\xi_{n}^{n+1})|^{2}\right]\nonumber \\
&\leqslant C\Delta t,
\end{align*}
where $L$ is the Lipschitz constant of $\sigma$, we deduce
\begin{align}\label{eq3.11}
&\mathbb E_{t_{n}}^{X^{n}}\left[|I_{d}-[\sigma^{n+1}]^{-1}A^{n}\sigma^{n}|^{2}\right] \nonumber\\
= &\mathbb E_{t_{n}}^{X^{n}}\left[\left|[\sigma^{n+1}]^{-1}(\sigma^{n+1}-\sigma^{n})-[\sigma^{n+1}]^{-1}\left(\partial_{x}b^{n}\Delta t+\sum_{j=1}^{d}\partial_{x}\sigma_{j}^{n}\Delta W_{t_{n+1}}^{j}\right)\sigma^{n}\right|^{2}\right] \nonumber\\
\leqslant &C\mathbb E_{t_{n}}^{X^{n}}\left[|\sigma^{n+1}-\sigma^{n}|^{2}\right]+C\Delta t\nonumber \\
\leqslant& C\Delta t.
\end{align}
By the inequalities \myeqref{eq3.9}, \myeqref{eq3.10}, and \myeqref{eq3.11}, we obtain
\begin{align}\label{eq3.12}
(\Delta t)^{2}|e_{z}^{n}|^{2} &\leqslant C_2(\Delta t)^{3}\mathbb E_{t_{n}}^{X^{n}}\left[|e_{z}^{n+1}|^{2}+|e_{y}^{n+1}|^{2}\right] \nonumber\\
&\quad + 6\Delta t\left(\mathbb E_{t_{n}}^{X^{n}}\left[|e_{y}^{n+1}|^{2}\right]-\left|\mathbb E_{t_{n}}^{X^{n}}\left[e_{y}^{n+1}\right]\right|^{2}\right)+ 12\left(|R_{z}^{n}|^{2}+|R_{Z1}^{n}|^{2}\right).
\end{align}
Dividing both sides of the inequality \myeqref{eq3.12} by $6\Delta t$ leads to
\begin{align}\label{eq3.13}
\frac{1}{6}\Delta t|e_{z}^{n}|^{2} &\leqslant C_2(\Delta t)^{2}\mathbb E_{t_{n}}^{X^{n}}\left[|e_{z}^{n+1}|^{2}+|e_{y}^{n+1}|^{2}\right] \nonumber\\
&+ \mathbb E_{t_{n}}^{X^{n}}\left[|e_{y}^{n+1}|^{2}\right]-\left|\mathbb E_{t_{n}}^{X^{n}}\left[e_{y}^{n+1}\right]\right|^{2}+ \frac{2}{\Delta t}\left(|R_{Z}^{n}|^{2}+|R_{Z1}^{n}|^{2}\right).
\end{align}
Similarily, 
\begin{align}\label{eq3.14}
\frac{1}{6}\Delta t\mathbb E_{t_{n}}^{X^{n}}\left[|e_{z}^{n+1}|^{2}\right] &\leqslant C_2(\Delta t)^{2}\mathbb E_{t_{n}}^{X^{n}}\left[|e_{z}^{n+2}|^{2}+|e_{y}^{n+2}|^{2}\right] \nonumber\\
&+ \mathbb E_{t_{n}}^{X^{n}}\left[|e_{y}^{n+2}|^{2}\right]-\left|\mathbb E_{t_{n}}^{X^{n}}\left[e_{y}^{n+2}\right]\right|^{2}+ \frac{2}{\Delta t}\mathbb E_{t_{n}}^{X^{n}}\left[|R_{Z}^{n+1}|^{2}+|R_{Z1}^{n+1}|^{2}\right].
\end{align} 
By \((1+\gamma \Delta t)	\lvert\mathbb{E}_{t_n}^{X^n} \left[e_y^{n+2}\right]\rvert^{2}+\mathbb E_{t_{n}}^{X^{n}}\left[|e_{y}^{n+2}|^{2}\right]-\left|\mathbb E_{t_{n}}^{X^{n}}\left[e_{y}^{n+2}\right]\right|^{2}\leq(1+\gamma \Delta t)	\mathbb{E}_{t_n}^{X^n} \left[\lvert e_y^{n+2}\rvert^{2}\right],\) we have

\begin{align}\label{eq3.15}
&	(1-\frac{C_1\Delta t}{\gamma })	\lvert	e_y^n \rvert^2+	(\frac{1}{6}-\frac{C_1}{\gamma })	\Delta t\lvert	e_z^n \rvert^2 +(\frac{1}{6}-\frac{C_1}{\gamma })\Delta t\mathbb E_{t_{n}}^{X^{n}}\left[|e_{z}^{n+1}|^{2}\right]\nonumber\\	 \leq&(1+\frac{C_1\Delta t}{\gamma })\mathbb E_{t_{n}}^{X^{n}}\left[|e_{y}^{n+1}|^{2}\right]+  (1+\gamma \Delta t+\frac{C_1\Delta t}{\gamma })\mathbb{E}_{t_n}^{X^n} \left[	\lvert e_y^{n+2}\rvert^{2}\right] +\frac{C_1\Delta t}{\gamma } \mathbb E_{t_n}^{X^n}\left[\lvert	e_z^{n+2} \rvert^2\right] \nonumber\\	 	 +&  (C_1+2C_2)(\Delta t)^2 \left\{\lvert	e_y^n \rvert^2+\lvert	e_z^n \rvert^2+ \mathbb E_{t_n}^{X^n}\left[\lvert	e_y^{n+1} \rvert^2+\lvert	e_z^{n+1} \rvert^2+\lvert	e_y^{n+2} \rvert^2+\lvert	e_z^{n+2} \rvert^2\right]\right\} 
\nonumber\\	+&C_1(1+\frac{1}{\gamma \Delta t})\left\{\lvert R_Y^{n}+\mathbb{E}_{t_n}^{X^n} \left[\tilde R_Y^{n+1}\right]\rvert^2+ \lvert R_{Y1}^{n}\rvert^2+ \lvert\mathbb{E}_{t_n}^{X^n} \left[\tilde R_{Y1}^{n+1}\right]\rvert^2\right\}\nonumber\\	+&  \frac{2}{\Delta t}\left(|R_{Z}^{n}|^{2}+|R_{Z1}^{n}|^{2}\right)+ \frac{2}{\Delta t}\mathbb E_{t_{n}}^{X^{n}}\left[|R_{Z}^{n+1}|^{2}+|R_{Z1}^{n+1}|^{2}\right].
\end{align}	
Consider
\begin{align}\label{eq3.16}
E_{t_n}^{X^n}\left[\lvert	e_y^{n+1} \rvert^2\right]& \leq (1+\gamma \Delta t)\mathbb{E}_{t_n}^{X^n} \left[	\lvert\mathbb{E}_{t_{n+1} }^{X^{n+1} } \left[e_y^{n+3}\right]\rvert^{2}\right]\nonumber\\	& +\frac{C_1\Delta t}{\gamma }  \mathbb E_{t_n}^{X^n}\left[\lvert	e_y^{n+1} \rvert^2+\lvert	e_z^{n+1} \rvert^2+\lvert	e_y^{n+2} \rvert^2+\lvert	e_z^{n+2} \rvert^2+\lvert	e_y^{n+3} \rvert^2+\lvert	e_z^{n+3} \rvert^2\right] \nonumber\\	& +C_1(\Delta t)^2 \mathbb E_{t_n}^{X^n}\left[\lvert	e_y^{n+1} \rvert^2+\lvert	e_z^{n+1} \rvert^2+\lvert	e_y^{n+2} \rvert^2+\lvert	e_z^{n+2} \rvert^2+\lvert	e_y^{n+3} \rvert^2+\lvert	e_z^{n+3} \rvert^2\right]
\nonumber\\	&+C_1(1+\frac{1}{\gamma \Delta t})\left\{ \mathbb{E}_{t_n}^{X^n} \left[\lvert\tilde R_Y^{n+1}+ R_Y^{n+2}\rvert^2\right]+ \mathbb{E}_{t_n}^{X^n} \left[\lvert\tilde R_{Y1}^{n+1}\lvert^2+ \lvert R_{Y1}^{n+2}\rvert^2\right]\right\},
\end{align}
\begin{align}\label{eq3.17}
\frac{1}{6}\Delta t\mathbb E_{t_{n}}^{X^{n}}\left[E_{t_{n+1}}^{X^{n+1}}\left[|e_{z}^{n+2}|^{2}\right]\right] &\leqslant C_2(\Delta t)^{2}\mathbb E_{t_{n}}^{X^{n}}\left[|e_{z}^{n+3}|^{2}+|e_{y}^{n+3}|^{2}\right] \nonumber\\
&+ \mathbb E_{t_{n}}^{X^{n}}\left[|e_{y}^{n+3}|^{2}\right]-E_{t_{n}}^{X^{n}}\left[\left|\mathbb E_{t_{n+1}}^{X^{n+1}}\left[e_{y}^{n+3}\right]\right|^{2}\right]+ \frac{2}{\Delta t}\mathbb E_{t_{n}}^{X^{n}}\left[|R_{Z}^{n+2}|^{2}+|R_{Z1}^{n+2}|^{2}\right].
\end{align} 
Plus \myeqref{eq3.16} and \myeqref{eq3.17} into \myeqref{eq3.15}, we have
\begin{align}\label{eq3.18}
&	(1-\frac{C_1\Delta t}{\gamma })	\lvert	e_y^n \rvert^2+	(\frac{1}{6}-\frac{C_1}{\gamma })	\Delta t\lvert	e_z^n \rvert^2 +(\frac{1}{6}-\frac{2C_1}{\gamma })\Delta t\mathbb E_{t_{n}}^{X^{n}}\left[|e_{z}^{n+1}|^{2}\right]\nonumber\\	 +&	E_{t_n}^{X^n}\left[\lvert	e_y^{n+1} \rvert^2\right]+(\frac{1}{6}-\frac{2C_1}{\gamma })\Delta t\mathbb E_{t_{n}}^{X^{n}}\left[|e_{z}^{n+2}|^{2}\right]\nonumber\\	 \leq&  (1+\gamma \Delta t+\frac{2C_1\Delta t}{\gamma })\mathbb{E}_{t_n}^{X^n} \left[	\lvert e_y^{n+2}\rvert^{2}\right] +(1+\gamma \Delta t+\frac{C_1\Delta t}{\gamma } ) \mathbb E_{t_{n}}^{X^{n}}\left[|e_{y}^{n+3}|^{2}\right] +\frac{C_1\Delta t}{\gamma }  \mathbb E_{t_n}^{X^n}\left[\lvert	e_z^{n+3} \rvert^2\right] \nonumber\\	+&(1+\frac{2C_1\Delta t}{\gamma })\mathbb E_{t_{n}}^{X^{n}}\left[|e_{y}^{n+1}|^{2}\right]\nonumber\\
+&2(C_1+C_2)(\Delta t)^2 \left\{\lvert	e_y^n \rvert^2+\lvert	e_z^n \rvert^2+ \mathbb E_{t_n}^{X^n}\left[\lvert	e_y^{n+1} \rvert^2+\lvert	e_z^{n+1} \rvert^2+\lvert	e_y^{n+2} \rvert^2+\lvert	e_z^{n+2} \rvert^2+|e_{z}^{n+3}|^{2}+|e_{y}^{n+3}|^{2}\right]\right\} 
\nonumber\\	+&C_1(1+\frac{1}{\gamma \Delta t})\left\{\lvert R_Y^{n}+\mathbb{E}_{t_n}^{X^n} \left[\tilde R_Y^{n+1}\right]\rvert^2+ \lvert R_{Y1}^{n}\rvert^2+ \lvert\mathbb{E}_{t_n}^{X^n} \left[\tilde R_{Y1}^{n+1}\right]\rvert^2\right\}\nonumber\\	+&C_1(1+\frac{1}{\gamma \Delta t})\left\{ \mathbb{E}_{t_n}^{X^n} \left[\lvert\tilde R_Y^{n+1}+ R_Y^{n+2}\rvert^2\right]+ \mathbb{E}_{t_n}^{X^n} \left[\lvert\tilde R_{Y1}^{n+1}\rvert^2+ \lvert R_{Y1}^{n+2}\rvert^2\right]\right\} \nonumber\\	+& \frac{2}{\Delta t}\left(|R_{Z}^{n}|^{2}+|R_{Z1}^{n}|^{2}\right)+ \frac{2}{\Delta t}\mathbb E_{t_{n}}^{X^{n}}\left[|R_{Z}^{n+1}|^{2}+|R_{Z1}^{n+1}|^{2}+|R_{Z}^{n+2}|^{2}+|R_{Z1}^{n+2}|^{2}\right].
\end{align}
Note that \begin{align}\label{eq3.19}
\Delta tE_{t_n}^{X^n}\left[\lvert	e_y^{n+1} \rvert^2\right]& \leq 	\Delta t\mathbb{E}_{t_n}^{X^n} \left[	\lvert e_y^{n+3}\rvert^{2}\right]\nonumber\\	& +(\gamma+\frac{C_1}{\gamma })(\Delta t)^2 \mathbb E_{t_n}^{X^n}\left[\lvert	e_y^{n+1} \rvert^2+\lvert	e_z^{n+1} \rvert^2+\lvert	e_y^{n+2} \rvert^2+\lvert	e_z^{n+2} \rvert^2+\lvert	e_y^{n+3} \rvert^2+\lvert	e_z^{n+3} \rvert^2\right]
\nonumber\\	&+\frac{C_1}{\gamma }\left\{ \mathbb{E}_{t_n}^{X^n} \left[\lvert\tilde R_Y^{n+1}+ R_Y^{n+2}\rvert^2\right]+ \mathbb{E}_{t_n}^{X^n} \left[\lvert\tilde R_{Y1}^{n+1}\rvert^2+ \lvert R_{Y1}^{n+2}\rvert^2\right]\right\},
\end{align}
Substituting into \myeqref{eq3.18}, we have
\begin{align}\label{eq3.20}
&	(1-\frac{C_1\Delta t}{\gamma })	\lvert	e_y^n \rvert^2+	(\frac{1}{6}-\frac{C_1}{\gamma })	\Delta t\lvert	e_z^n \rvert^2 +(\frac{1}{6}-\frac{2C_1}{\gamma })\Delta t\mathbb E_{t_{n}}^{X^{n}}\left[|e_{z}^{n+1}|^{2}\right] +	(\frac{1}{6}-\frac{2C_1}{\gamma })\Delta t\mathbb E_{t_{n}}^{X^{n}}\left[|e_{z}^{n+2}|^{2}\right]\nonumber\\	 \leq&  (1+\gamma \Delta t+\frac{2C_1\Delta t}{\gamma })\mathbb{E}_{t_n}^{X^n} \left[	\lvert e_y^{n+2}\rvert^{2}\right] +(1+\gamma \Delta t+\frac{3C_1\Delta t}{\gamma } ) \mathbb E_{t_{n}}^{X^{n}}\left[|e_{y}^{n+3}|^{2}\right] +\frac{C_1\Delta t}{\gamma }  \mathbb E_{t_n}^{X^n}\left[\lvert	e_z^{n+3} \rvert^2\right] 
\nonumber\\	+&\frac{C_1}{\gamma \Delta t}\left\{\lvert R_Y^{n}+\mathbb{E}_{t_n}^{X^n} \left[\tilde R_Y^{n+1}\right]\rvert^2+ \lvert R_{Y1}^{n}\rvert^2+\lvert \mathbb{E}_{t_n}^{X^n} \left[\tilde R_{Y1}^{n+1}\right]\rvert^2\right\}\nonumber\\	+&\frac{C_1}{\gamma \Delta t}\left\{ \mathbb{E}_{t_n}^{X^n} \left[\lvert\tilde R_Y^{n+1}+ R_Y^{n+2}\rvert^2\right]+ \mathbb{E}_{t_n}^{X^n} \left[\lvert\tilde R_{Y1}^{n+1}\rvert^2+ \lvert R_{Y1}^{n+2}\rvert^2\right]\right\} \nonumber\\	+& \frac{2}{\Delta t}\left(|R_{Z}^{n}|^{2}+|R_{Z1}^{n}|^{2}\right)+ \frac{2}{\Delta t}\mathbb E_{t_{n}}^{X^{n}}\left[|R_{Z}^{n+1}|^{2}+|R_{Z1}^{n+1}|^{2}+|R_{Z}^{n+2}|^{2}+|R_{Z1}^{n+2}|^{2}\right]+O((\Delta t)^2).
\end{align}
For any \(1 \leq k \leq N-n-1,\) by induction,
\begin{align}\label{eq3.21}
&	(1-\frac{C_1\Delta t}{\gamma })	\lvert	e_y^n \rvert^2+	(\frac{1}{6}-\frac{C_1}{\gamma })	\Delta t\lvert	e_z^n \rvert^2 +(\frac{1}{6}-\frac{2C_1}{\gamma })\Delta t\mathbb E_{t_{n}}^{X^{n}}\left[|e_{z}^{n+1}|^{2}+|e_{z}^{n+k}|^{2}\right] 
\nonumber\\	+&  (\frac{1}{6}-\frac{3C_1}{\gamma })\Delta t\sum\limits_{i=2}^{k-1}\mathbb E_{t_{n}}^{X^{n}}\left[|e_{z}^{n+i}|^{2}\right]
\nonumber\\	 \leq&  (1+k\gamma \Delta t+\frac{3kC_1\Delta t}{\gamma } ) \mathbb E_{t_{n}}^{X^{n}}\left[	\lvert e_y^{n+k}\rvert^{2}+|e_{y}^{n+k+1}|^{2}\right] +\frac{C_1\Delta t}{\gamma }  \mathbb E_{t_n}^{X^n}\left[\lvert	e_z^{n+k+1} \rvert^2\right] 
\nonumber\\		+&\frac{C_1}{\gamma \Delta t}\sum\limits_{i=0}^{k-1}\left\{ \mathbb{E}_{t_n}^{X^n} \left[\lvert R_Y^{n+i}+\tilde R_Y^{n+i+1}\rvert^2+\lvert\tilde R_Y^{n+i}+ R_Y^{n+i+1}\rvert^2\right] \right\}
\nonumber\\		+&\frac{C_1}{\gamma \Delta t}\sum\limits_{i=0}^{k}\left\{ \mathbb{E}_{t_n}^{X^n} \left[\lvert R_{Y1}^{n+i}\rvert^2+ \lvert\tilde R_{Y1}^{n+i}\rvert^2\right]\right\} + \frac{2}{\Delta t}\sum\limits_{i=0}^{k}\mathbb E_{t_{n}}^{X^{n}}\left[|R_{Z}^{n+i}|^{2}+|R_{Z1}^{n+i}|^{2}\right]+O((\Delta t)^2),
\end{align}

Set \(k:=N-n-1,\,\gamma: = 18C_1,C:=1+18C_{1}T+\frac{T}{6}+\frac{\Delta t}{18}\), by using\(\frac{1}{1-\frac{\Delta t}{18}}=1+\frac{\Delta t}{18}+O((\Delta t)^2 ),\) we have
\begin{align}\label{eq3.22}
&	\lvert	e_y^n \rvert^2+	\frac{1}{9}	\Delta t\lvert	e_z^n \rvert^2 \nonumber\\	 \leq&  C\mathbb{E}_{t_n}^{X^n} \left[	\lvert e_y^{N-1}\rvert^{2}+\lvert e_y^{N}\rvert^{2}\right] +\frac{\Delta t}{18} \mathbb E_{t_n}^{X^n}\left[\lvert	e_z^{N} \rvert^2\right] \nonumber\\	 +&\frac{2}{ \Delta t}\sum\limits_{i=n}^{N-2}\left\{ \mathbb{E}_{t_n}^{X^n} \left[\lvert R_Y^{i}+\tilde R_Y^{i+1}\rvert^2+\lvert\tilde R_Y^{i}+ R_Y^{i+1}\rvert^2\right]\right\} \nonumber\\		+&  \frac{2}{\Delta t}\sum\limits_{i=n}^{N-1}\mathbb E_{t_{n}}^{X^{n}}\left[\lvert R_{Y1}^{i}\rvert^2+\lvert\tilde R_{Y1}^{i}\rvert^2+|R_{Z}^{i}|^{2}+|R_{Z1}^{i}|^{2}\right]+O((\Delta t)^2).
\end{align}
From \citep{ZhaoLi2014} and \myeqref{zn2}, it can be seen that the error estimation of Scheme 2.2 proves to be similar. The proof is completed.

\end{proof}

\subsection{Error estimate}

\begin{theorem}\label{Theorem3.2}
Let $\left(X_t^{t_n, X^n}, Y_t^{t_n, X^n}, Z_t^{t_n, X^n}\right)_{t_n \leq t \leq T}(n = 0, 1, \ldots, N)$ be the solutions to the FBSDEs \myeqref{eq2.2}  and $\left(X^n, Y^n, Z^n\right)(n = 0, 1, \ldots, N)$  be the solutions to Scheme 2.1 or Scheme 2.2 with $Y^{N} = \Phi(X^{N})$ and $Z^{N} = \Phi_{x}(X^{N})\sigma(t_{N}, X^{N})$. 
\noindent	(1.) Assume \( \Phi \in C_{b}^{4,\alpha}(\alpha \in (0,1]),f \in  C_{b}^{2,4,4,4}\,b,\,\sigma \in C_{b}^{2,4},\) and the matrix \(\left|\sigma^{-1}\right|\) is bounded. Then under \sref{Hypothesis}{H2.1} and \sref{Hypothesis}{H2.2}, we have:
\begin{equation}\label{eq3.23}
\max\limits_{0\leq n \leq N} \left(\left|Y_{t_n}^{t_n, X^n} - Y^n\right|^2 + \Delta t \left|Z_{t_n}^{t_n, X^n} - Z^n\right|^2 \right) \leq C (\Delta t)^{\min\{2,2\gamma\}}.
\end{equation}
\noindent (2.)Assume \sref{Hypothesis}{H2.2} holds, 
\( \Phi \in C_{b}^{6,\alpha}(\alpha \in (0,1]),\,f_{1},\,f_{2} \in  C_{b}^{3,6,6,6},\,b,\,\sigma \in C_{b}^{3,6},\) and the matrix \(\left|\sigma^{-1}\right|\) is bounded. Then for \( n = 0, 1, \ldots, N\), we have the estimates
\begin{equation}\label{eq3.24}
\max\limits_{0\leq n \leq N} \left(\left|Y_{t_n}^{t_n, X^n} - Y^n\right|^2 + \Delta t \left|Z_{t_n}^{t_n, X^n} - Z^n\right|^2 \right) \leq C (\Delta t)^{\min\{4,2\gamma\}}.
\end{equation}
where $C>0$ is independent of $\Delta t$.
\end{theorem}
To prove the above theorem, we introduce \sref{Lemma}{lemma3.1}.

\begin{lemma}\label{lemma3.1}(See \citep{ZhaoLi2014}).  Let \(  R^n_{Y1},\,\tilde R^n_{Y1},\,R^n_{Z},\,R^n_{Z1},\,R^n_{Z2}\) be the truncation errors defined in \myeqref{Rny1}, \myeqref{tRny1}, \myeqref{Rnz}, \myeqref{Rnz1}, and \myeqref{Rnz2}, respectively. 

\noindent	(1.) Assume \sref{Hypothesis}{H2.2} holds, \( \Phi \in C_{b}^{4,\alpha}(\alpha \in (0,1]),\,f_{1},\,f_{2} \in  C_{b}^{2,4,4,4},\,b,\,\sigma \in C_{b}^{2,4},\) and the matrix \(\left|\sigma^{-1}\right|\) is bounded. Then for \( n = 0, 1, \ldots, N\), we have the estimates
\begin{align*}
&	\left|R^{n}_{Y}\right| \leq C (\Delta t)^{2},\,\left|R^{n}_{Y1}\right| \leq C (\Delta t)^{2},\quad  \left|\tilde R^{n}_{Y1}\right|  \leq C (\Delta t)^{2},\nonumber\\ & \left|R^{n}_{Z}\right|  \leq C (\Delta t)^{\min\{2,\gamma+2\}},\,\left|R^{n}_{Z1}\right|  \leq C (\Delta t)^{\min\{2,\gamma+1\}},\,\left|R^{n}_{Z2}\right|  \leq C (\Delta t)^{\min\{2,\gamma+1\}}.
\end{align*} 

\noindent (2.)Assume \sref{Hypothesis}{H2.2} holds, 
\( \Phi \in C_{b}^{6,\alpha}(\alpha \in (0,1]),\,f_{1},\,f_{2} \in  C_{b}^{3,6,6,6},\,b,\,\sigma \in C_{b}^{3,6},\) and the matrix \(\left|\sigma^{-1}\right|\) is bounded. Then for \( n = 0, 1, \ldots, N\), we have the estimates
\begin{align*}
&\left|R^{n}_{Y}\right| \leq C (\Delta t)^{2},\,\left|R^{n}_{Y1}\right| \leq C (\Delta t)^{3},\quad  \left|\tilde R^{n}_{Y1}\right|  \leq C (\Delta t)^{3},\nonumber\\ & \left|R^{n}_{Z}\right|  \leq C (\Delta t)^{\min\{3,\gamma+2\}},\,\left|R^{n}_{Z1}\right| \leq C (\Delta t)^{\min\{3,\gamma+1\}},\,\left|R^{n}_{Z2}\right|  \leq C (\Delta t)^{\min\{3,\gamma+1\}}.
\end{align*}  
\end{lemma}

\begin{lemma}\label{lemma3.2} Let \(  R^n_{Y},\,\tilde R^n_{Y}\) be the truncation errors defined in \myeqref{Rny} and \myeqref{tRny}, respectively. 

\noindent	(1.)Assume \sref{Hypothesis}{H2.2} holds, \( \Phi \in C_{b}^{4,\alpha}(\alpha \in (0,1]),\,f_{1},\,f_{2} \in  C_{b}^{2,4,4,4},\,b,\,\sigma \in C_{b}^{2,4},\) and the matrix \(\left|\sigma^{-1}\right|\) is bounded. Then for \( n = 0, 1, \ldots, N\), we have the estimates
\begin{equation*}
\left|R^{n}_{Y}+\mathbb E_{t_n}^{X^n}[\tilde R^{n+1}_{Y}]\right|  \leq C (\Delta t)^{2}, 	\left|\tilde R^{n}_{Y}+\mathbb E_{t_n}^{X^n}[ R^{n+1}_{Y}]\right|  \leq C (\Delta t)^{2}.
\end{equation*} 

\noindent (2.)Assume \sref{Hypothesis}{H2.2} holds, 
\( \Phi \in C_{b}^{6,\alpha}(\alpha \in (0,1]),\,f_{1},\,f_{2} \in  C_{b}^{3,6,6,6},\,b,\,\sigma \in C_{b}^{3,6},\) and the matrix \(\left|\sigma^{-1}\right|\) is bounded. Then for \( n = 0, 1, \ldots, N\), we have the estimates
\begin{equation*}
\left|R^{n}_{Y}+\mathbb E_{t_n}^{X^n}[\tilde R^{n+1}_{Y}]\right|  \leq C (\Delta t)^{\min\{3,\gamma+1\}},\,	\left|\tilde R^{n}_{Y}+\mathbb E_{t_n}^{X^n}[ R^{n+1}_{Y}]\right|  \leq C (\Delta t)^{\min\{3,\gamma+1\}}.
\end{equation*} 
\end{lemma}

\begin{proof}
\noindent	(1.)	Without loss of generality, we assume \(n\) is an even number.	By \myeqref{Rny} and \myeqref{tRny}, we have
\begin{align}\label{eq3.25}
&	R_Y^{n}+\mathbb E_{t_n}^{X^n}[\tilde R^{n+1}_{Y}]\nonumber\\ =& \int_{t_n}^{t_{n+1}}  \mathbb E_{t_n}^{X^n}[f_s^{t_n, X^n}-f_{2,t_{n+1}}^{t_n, X^n}- f_{1,t_n}^{t_n, X^n}]   ds+\int_{t_{n+1}}^{t_{n+2}}  \mathbb E_{t_n}^{X^n}[f_s^{t_{n+1}, X^{n+1}}-f_{1,t_{n+2}}^{t_{n+1}, X^{n+1}} - f_{2,t_{n+1}}^{t_{n+1}, X^{n+1}}]  ds\nonumber\\ 
=&\int_{t_n}^{t_{n+1}}  \mathbb E_{t_n}^{X^n}[f_s^{t_n, X^n}-\frac{1}{2}(f_{t_{n+1}}^{t_n, X^n}+f_{t_n}^{t_n, X^n})]  ds+ \int_{t_{n+1}}^{t_{n+2}}  \mathbb E_{t_n}^{X^n}[f_s^{t_{n+1}, X^{n+1}}-\frac{1}{2}(f_{t_{n+2}}^{t_{n+1}, X^{n+1}} + f_{t_{n+1}}^{t_{n+1}, X^{n+1}})]  ds\nonumber\\ 
+&\Delta t \mathbb E_{t_n}^{X^n}[\frac{1}{2}f_{1,t_{n+1}}^{t_n, X^n}+\frac{1}{2} f_{2,t_n}^{t_n, X^n}-\frac{1}{2}f_{2,t_{n+1}}^{t_n, X^n}-\frac{1}{2} f_{1,t_n}^{t_n, X^n}       ]\nonumber\\ 
+&\Delta t \mathbb E_{t_n}^{X^n}[  \frac{1}{2}f_{2,t_{n+2}}^{t_{n+1}, X^{n+1}} + \frac{1}{2}f_{1,t_{n+1}}^{t_{n+1}, X^{n+1}}  -\frac{1}{2}f_{1,t_{n+2}}^{t_{n+1}, X^{n+1}} - \frac{1}{2}f_{2,t_{n+1}}^{t_{n+1}, X^{n+1}}] \nonumber\\ =&I+\Delta t \mathbb E_{t_n}^{X^n}[g_{t_{n}}^{t_n, X^n}-g_{t_{n+1}}^{t_n, X^n}  + g_{t_{n+2}}^{t_{n+1}, X^{n+1}} - g_{t_{n+1}}^{t_{n+1}, X^{n+1}} ]
, 
\end{align}
where \(I:=\int_{t_n}^{t_{n+1}}  \mathbb E_{t_n}^{X^n}[f_s^{t_n, X^n}-\frac{1}{2}(f_{t_{n+1}}^{t_n, X^n}+f_{t_n}^{t_n, X^n})]  ds+ \int_{t_{n+1}}^{t_{n+2}}  \mathbb E_{t_n}^{X^n}[f_s^{t_{n+1}, X^{n+1}}-\frac{1}{2}(f_{t_{n+2}}^{t_{n+1}, X^{n+1}} + f_{t_{n+1}}^{t_{n+1}, X^{n+1}})]  ds,\,g := \frac{1}{2}(f_{2}-f_{1}).\)
Then by \citep{ZhaoWang2009}, we obtain
\begin{equation}\label{eq3.26}
\left|I\right|\leq C (\Delta t)^{2}.\end{equation}
Define \[G_{1}(s): = \mathbb E_{t_n}^{X^n}[g(s,X_{s}^{t_n, X^n})],\,s \in [t_{n},\,t_{n+1}],\]
\[G_{2}(s): = \mathbb E_{t_n}^{X^n}[g(s,X_{s}^{t_{n+1}, X^{n+1}})],\,s \in [t_{n+1},\,t_{n+2}].\]
By using \sref{Hypothesis}{H2.2}, we have \(\left|G'_{2}(t_{n+1})-G'_{1}(t_{n+1})\right|\leq C (\Delta t)^{2}.\)
Therefore,
\begin{align}\label{eq3.27}
&\mathbb E_{t_n}^{X^n}[g_{t_{n}}^{t_n, X^n}-g_{t_{n+1}}^{t_n, X^n}  + g_{t_{n+2}}^{t_{n+1}, X^{n+1}} - g_{t_{n+1}}^{t_{n+1}, X^{n+1}}] \nonumber\\ =&G_{1}(t_n)-G_{1}(t_{n+1})+G_{2}(t_{n+2})-G_{2}(t_{n+1})
\nonumber\\ =&G'_{1}(t_{n+1})(-\Delta t)+\frac{1}{2}G''_{1}(\eta_{1})(\Delta t)^{2}+G'_{2}(t_{n+1})\Delta t+\frac{1}{2}G''_{2}(\eta_{2})(\Delta t)^{2} \leq C (\Delta t)^2.
\end{align}
Substituting \myeqref{eq3.25} and \myeqref{eq3.26} into \myeqref{eq3.24}, we obtain \(	\left|R^{n}_{Y}+\mathbb E_{t_n}^{X^n}[\tilde R^{n+1}_{Y}]\right|  \leq C (\Delta t)^{2}.\) Another estimate and (2.) can be proved similarly. 
The proof is completed.
\end{proof}
\begin{proof}[Proof of \sref{Theorem}{Theorem3.2}]
(1.)	Let $\left(X_t^{t_n, X^n}, Y_t^{t_n, X^n}, Z_t^{t_n, X^n}\right)_{t_n \leq t \leq T}(n = 0, 1, \ldots, N)$ be the solutions to the FBSDEs \myeqref{eq2.2}  and $\left(X^n, Y^n, Z^n\right)(n = 0, 1, \ldots, N)$  be the solutions to Scheme 2.1 or Scheme 2.2 with $Y^{N} = \Phi(X^{N})$ and $Z^{N} = \Phi_{x}(X^{N})\sigma(t_{N}, X^{N})$. Under the conditions of the theorem, \sref{Theorem}{Theorem3.1}, \sref{Lemma}{lemma3.1}(1.) and \sref{Lemma}{lemma3.2}(1.) hold. 
Therefore, by using \myeqref{eq3.1}, we have
\begin{align}\label{eq3.28}
& \left|Y_{t_n}^{t_n, X^n} - Y^n\right|^2 + \frac{1}{9} \Delta t \left|Z_{t_n}^{t_n, X^n} - Z^n\right|^2 \nonumber\\
& \leq C \mathbb{E}_{t_n}^{X^n} \left[  \left|Y_{t_{N-1}}^{t_{N-1}, X^{N-1}} - Y^{N-1}\right|^2 \right] + C (\Delta t)^{\min\{3,2\gamma+1\}}N\nonumber\\
& \leq C \mathbb{E}_{t_n}^{X^n} \left[  \left|Y_{t_{N-1}}^{t_{N-1}, X^{N-1}} - Y^{N-1}\right|^2 \right] + CT (\Delta t)^{\min\{2,2\gamma\}}.
\end{align}
%
By \myeqref{eq3.13} and \sref{Lemma}{lemma3.1}(1.), we obtain
\begin{align*}
\Delta t|e_{z}^{N-1}|^{2} &\leqslant C(\Delta t)^{2}\mathbb E_{t_{N-1}}^{X^{N-1}}\left[|e_{z}^{N}|^{2}+|e_{y}^{N}|^{2}\right] + \frac{2}{\Delta t}\left(|R_{Z}^{N-1}|^{2}+|R_{Z1}^{N-1}|^{2}\right) \le C(\Delta t)^{\min\{3,2\gamma+1\}}.
\end{align*}
Then from	\myeqref{eq3.2}, we have
\begin{align}\label{eq3.29}
\left|e^{N-1}_{y}\right|^{2} &\leq C \left\{ \mathbb{E}_{t_{N-1}}^{X^{N-1}} \left[\left|e_y^{N}\right|^2\right]+  (\Delta t)^2 \left|e^{N-1}_{z} \right|^2+(\Delta t)^2 \mathbb E_{t_{N-1}}^{X^{N-1}}[\left|e^{N}_{z}\right|^2] + \left|R_Y^{N-1}\right|^2+ \left|R_{Y1}^{N-1} \right|^2\right\}\nonumber\\
&\le C(\Delta t)^{\min\{4,2\gamma+2\}}.	\end{align}
Substituting \myeqref{eq3.29} into \myeqref{eq3.28}, we obtain \[\left|Y_{t_n}^{t_n, X^n} - Y^n\right|^2 + \frac{1}{9} \Delta t \left|Z_{t_n}^{t_n, X^n} - Z^n\right|^2 \le C(\Delta t)^{\min\{2,2\gamma+2\}}.\]
The proof of (2.) is similar. We complete the proof.
\end{proof}

\section{Numerical tests}\label{sec4}
In this chapter, we will verify the accuracy and effectiveness of the proposed Schemes 2.1 and 2.2 through two numerical examples, and validate the theoretical results mentioned above.
First we need a subdivision of d-dimensional Euclidean space \(\mathbb R ^{d}:\)
\begin{equation}\label{eq4.1}
\mathcal R^{d}_{h}= \left\{ \mathtt x_{i} = (x_{i_{1}}^{1},\,x_{i_{2}}^{2},\,\cdots,\,x_{i_{d}}^{d})^{T}\mid  \mathtt x_{i} \in \mathbb{R}^{d},\,i_{1},\,\cdots,\,i_{d} \in \mathbb{Z}\right\},
\end{equation}
Here, \(\mathbb{Z}\) represents the set of all positive integers. For a fixed time step \(\Delta t = T/N\) , we use a uniform spatial step \(\Delta x = (\Delta t)^{\frac{p+1}{r+1}}\) in all directions, where (p) is the order of convergence in time, and (r) is the order of spatial interpolation. Therefore, in myeqref{eq4.1}, \(x_{i_{k}}^{k} = i_{k} \Delta x, \,k =1,\,\cdots,\,d.\)
We choose the Gauss–Hermite quadrature formula to approximate the conditional mathematical expectation that needs to be estimated, and we set the number of Gauss integration points to be sufficiently large. The values of the integrand required when calculating the conditional mathematical expectation are obtained through cubic spline interpolation. We denote the running time by RT and the convergence order obtained through linear least-square fitting by CR.

We use a terminal time of \((T=1.0.)\) Let\(((Y_{0},Z_{0}))\)  denote the true solution of the forward–backward stochastic differential equation myeqref{eq2.2} at the initial time, and \((Y^{0},Z^{0})\) denote the approximate solution at (n=0) obtained through an approximate solution of myeqref{eq2.2}.

We denote our Scheme 2.1 and 2.2 as S1 and S2, respectively.All numerical experiments were performed using CPU computations, and the software used Python 3.11.13.

\begin{example}\label{example4.1}(Convergence test)
Consider a one-dimensional Heston model proposed by \citep{Cerny2008} 
\begin{equation}\label{eq4.2}
\begin{cases}
S_t = S_0 + \int_{0}^{t} S_t (\mu X_t dt + \sqrt{X_{s}} dW^1_t),\\
X_t = X_0 + \int_{0}^{t} \kappa (\theta -X_s)ds + \int_{0}^{t} \sigma \sqrt{X_{s}} (\rho dW^1_s+\sqrt{1-\rho^2}dW^2_s),
\end{cases}
\end{equation}
where \( \mu,\,\kappa,\,\theta,\,\sigma>0, \, -1 \leq \rho \leq 1\) and \(2\kappa \theta \geq \sigma^2.\)
Consider the backward stochastic Riccati equation  
\begin{equation}\label{eq4.3}
Y_t = 1- \int_t^T \left( Y_s \mu^2X_s+2\mu \sqrt{X_s}Z^1_s+\frac{\left|Z^1_s\right|^2}{Y_s } \right)ds -\int_t^T Z^1_sdW^1_s- \int_t^T Z^2_sdW^2_s
\end{equation}
where \(Y_t\) is named as opportunity process in \citep{Cerny2008}, and they gave expression of \(Y\) as following:
\begin{equation}\label{eq4.4}
Y_t = \exp\{\chi^0(t)+\chi^1(t) X_t\},
\end{equation}
where \(\chi^0(t)\) and \(\chi^1(t)\)is defined at \(t\in [0,\,T]\) by
\begin{align*}
\chi_0(t) &=F \left( -\frac{ B}{2C}(T-t) 
- \frac{1}{C} \log \left( 
\frac{(B + D)e^{- D(T-t)/2} - ( B -  D)e^{ D(T-t)/2}}{2 D}
\right) \right), \\[6pt]
\chi_1(t) &= -\frac{ B}{2 C} 
+  \frac{ D}{2 C}  \frac{( B +  D)e^{- D(T-t)/2} + ( B -  D)e^{ D(T-t)/2}}
{( B +  D)e^{- D(T-t)/2} - ( B -  D)e^{ D(T-t)/2}},
\end{align*}
where
\begin{align*}
A &:= -\mu^2,  B := -\kappa - 2\rho\sigma\mu,
C :=  \frac{1}{2}\sigma^2(1-2\rho^2), \\
D &:= \sqrt{ B^2 - 4 A C},		 F := \kappa\theta,
\end{align*}
with \(C,D\neq 0.\)
For simplify，we use a one-dimensional Brownian motion \(dW_s = \rho dW^1_s+\sqrt{1-\rho^2}dW^2_s,\) and rewrite \myeqref{eq4.3} as
\[ d Y_s= \left( Y_s \mu^2X_s+2\mu \sqrt{X_s} \rho  Z_s+\frac{\left|\rho Z_s\right|^2}{Y_s } \right)ds + \tilde Z_sdW_s.\]
Then we have \[ Z^1_s = \rho  Z_s,\, Z^2_s = \sqrt{1-\rho^2} Z_s.\]
We set \(f_1(s,x,y,z) :=-y \mu^2x-2\mu \sqrt{x}\rho z,\,f_2(s,x,y,z):=-\frac{\left|\rho z\right|^2}{y}.  \mu=0.3,\,\kappa=0.5,\,\theta=0.8,\,\sigma=0.2 \rho =0.8. X_0=1.0.\)
To compare the convergence rate dominated by the forward SDE discretization, we vary the forward schemes among Euler, Milstein, and the weak order-2 It\^{o}--Taylor method. We use \myeqref{eq4.4} compute the explicit solution.	For each $N$, we compute the errors in $Y_0$ and $Z_0$ with respect to a reference solution obtained on a very fine grid. The errors \(\left|Y_{0}-Y^0\right|\) and \(\left|Z_{0}-Z^0\right|\), toghther with convergence rate are in table 1, in which we use WT2 to denote Weak order-2.0 scheme. The convergence rates of the two schemes are of order 1 if the Euler scheme or the Milstein scheme is used for solving the SDE, and are of order 2 if the order-2 weak It\^o-Taylor scheme is used for solving the SDE.
\begin{table}[htbp]
	\centering
	\caption{Numerical errors and observed convergence rates.}
	\label{tab:err_cr_no_rt}
	\setlength{\tabcolsep}{6pt}
	\renewcommand{\arraystretch}{1.15}
	
	\begin{tabular}{c cc cc cc}
		\toprule
		\multirow{2}{*}{$\Delta t$}
		& \multicolumn{2}{c}{Euler}
		& \multicolumn{2}{c}{Milstein}
		& \multicolumn{2}{c}{WT2} \\
		\cmidrule(lr){2-3}\cmidrule(lr){4-5}\cmidrule(lr){6-7}
		& $|Y_0-Y^0|$ & $|Z_0-Z^0|$
		& $|Y_0-Y^0|$ & $|Z_0-Z^0|$
		& $|Y_0-Y^0|$ & $|Z_0-Z^0|$ \\
		\midrule
		\multicolumn{7}{l}{\textbf{S1 (Scheme 2.1)}} \\
		\midrule
		$1/16$  & $8.25{\times}10^{-5}$ & $1.62{\times}10^{-4}$
		& $8.23{\times}10^{-5}$ & $1.61{\times}10^{-4}$
		& $1.08{\times}10^{-5}$ & $7.91{\times}10^{-6}$ \\
		$1/32$  & $4.32{\times}10^{-5}$ & $8.10{\times}10^{-5}$
		& $4.31{\times}10^{-5}$ & $8.05{\times}10^{-5}$
		& $2.70{\times}10^{-6}$ & $2.00{\times}10^{-6}$ \\
		$1/64$  & $2.21{\times}10^{-5}$ & $4.05{\times}10^{-5}$
		& $2.20{\times}10^{-5}$ & $4.03{\times}10^{-5}$
		& $6.77{\times}10^{-7}$ & $5.02{\times}10^{-7}$ \\
		$1/128$ & $1.12{\times}10^{-5}$ & $2.03{\times}10^{-5}$
		& $1.11{\times}10^{-5}$ & $2.01{\times}10^{-5}$
		& $1.69{\times}10^{-7}$ & $1.26{\times}10^{-7}$ \\
		\midrule
		CR
		& \multicolumn{1}{c}{\textbf{0.963}} & \multicolumn{1}{c}{\textbf{0.999}}
		& \multicolumn{1}{c}{\textbf{0.963}} & \multicolumn{1}{c}{\textbf{0.998}}
		& \multicolumn{1}{c}{\textbf{1.996}} & \multicolumn{1}{c}{\textbf{1.992}} \\
		\bottomrule
	\end{tabular}
	
	\vspace{0.8em}
	
	\begin{tabular}{c cc cc cc}
		\toprule
		\multirow{2}{*}{$N$}
		& \multicolumn{2}{c}{Euler}
		& \multicolumn{2}{c}{Milstein}
		& \multicolumn{2}{c}{WT2} \\
		\cmidrule(lr){2-3}\cmidrule(lr){4-5}\cmidrule(lr){6-7}
		& $|Y_0-Y^0|$ & $|Z_0-Z^0|$
		& $|Y_0-Y^0|$ & $|Z_0-Z^0|$
		& $|Y_0-Y^0|$ & $|Z_0-Z^0|$ \\
		\midrule
		\multicolumn{7}{l}{\textbf{S2 (Scheme 2.2)}} \\
		\midrule
		$16$  & $1.30{\times}10^{-4}$ & $3.62{\times}10^{-4}$
		& $1.29{\times}10^{-4}$ & $3.61{\times}10^{-4}$
		& $1.08{\times}10^{-5}$ & $8.18{\times}10^{-6}$ \\
		$32$  & $6.78{\times}10^{-5}$ & $1.83{\times}10^{-4}$
		& $6.76{\times}10^{-5}$ & $1.82{\times}10^{-4}$
		& $2.72{\times}10^{-6}$ & $2.07{\times}10^{-6}$ \\
		$64$  & $3.46{\times}10^{-5}$ & $9.17{\times}10^{-5}$
		& $3.46{\times}10^{-5}$ & $9.12{\times}10^{-5}$
		& $6.81{\times}10^{-7}$ & $5.19{\times}10^{-7}$ \\
		$128$ & $1.75{\times}10^{-5}$ & $4.59{\times}10^{-5}$
		& $1.75{\times}10^{-5}$ & $4.57{\times}10^{-5}$
		& $1.70{\times}10^{-7}$ & $1.30{\times}10^{-7}$ \\
		\midrule
		CR
		& \multicolumn{1}{c}{\textbf{0.963}} & \multicolumn{1}{c}{\textbf{0.993}}
		& \multicolumn{1}{c}{\textbf{0.963}} & \multicolumn{1}{c}{\textbf{0.994}}
		& \multicolumn{1}{c}{\textbf{1.996}} & \multicolumn{1}{c}{\textbf{1.992}} \\
		\bottomrule
	\end{tabular}
\end{table}

\end{example}

\begin{example}\label{example4.2}(See \citep{Zheng2025}) Consider the following two-dimensional FBSDE；
\begin{equation*}
\begin{cases}
X_t = X_0 + \sigma W_t,\\
Y_t = Y_T + \int_t^T \left(Y_t+ \frac{5}{2}\sigma^2 e^{-2t} \frac{Y_t}{Y_t^2 + (Z_t \tilde{\sigma})^2} \right)  ds - \int_t^T Z_s  dW_s,
\end{cases}
\end{equation*}
with $Y_T = e^{-T} \sin(X_T^1 + 2X_T^2)$. Here $Z_t = (Z_t^1, Z_t^2)$, $\tilde{\sigma} = (\frac{3}{\sigma}, -\frac{1}{\sigma})^\top$, and $W_t$ is a standard two-dimensional Brownian motion. The analytical solution is given by
\begin{equation*}
\begin{cases}
Y_t = e^{-t} \sin(X_t^1 + 2X_t^2), \\
Z_t = \left( \sigma e^{-t} \cos(X_t^1 + 2X_t^2),\; 2\sigma e^{-t} \cos(X_t^1 + 2X_t^2) \right).
\end{cases}
\end{equation*}
In this example, we set $\sigma = 0.2$. We set \(f_1(y)=y,\,f_2(t,x,y,z)= \frac{5}{2}\sigma^2 e^{-2t} \frac{y}{y^2 + (z \tilde{\sigma})^2}.\) In Tables 2 and 3, we denote Scheme 2.2 in \citep{ZhaoLi2014} as Z1, and Scheme 2.1 proposed in this paper as S1.
Table 2 reports the total numbers of Picard iterations required by the two implicit schemes under different time discretizations, while Table 3 presents the corresponding numerical errors and observed convergence orders.

From Table 3, it can be observed that S1 yields smaller errors for the $Y$, component but slightly larger errors for the $Z_1$ and $Z_2$ components compared with Z1.
Nevertheless, by sacrificing only a limited amount of accuracy, S1 significantly reduces the computational cost in the implicit stage, saving approximately 80\% of the Picard iterations, and thus achieves a more favorable accuracy–cost trade-off.
	\begin{table}[htbp]
	\centering
	\caption{Total number of inner Picard iterations for different $N$ (lower is better).}
	\label{tab:total_picard}
	\begin{tabular}{c|cc|c}
		\hline
		$N$ & Z1 & S1 & Ratio (S1/Z1) \\
		\hline
		8   & 30   & 8    & 0.267 \\
		16  & 84   & 56   & 0.667 \\
		32  & 310  & 55   & 0.177 \\
		64  & 810  & 275  & 0.339 \\
		128 & 2890 & 599  & 0.207 \\
		\hline
	\end{tabular}
\end{table}

\begin{table}[htbp]
	\centering
	\caption{$L^2$ errors and observed convergence orders of $Y$, $Z_1$ and $Z_2$.}
	\label{tab:accuracy}
	\setlength{\tabcolsep}{6pt}
	\begin{tabular}{c|ccc|ccc}
		\hline
		\multirow{2}{*}{$N$} 
		& \multicolumn{3}{c|}{Z1} 
		& \multicolumn{3}{c}{S1} \\
		\cline{2-7}
		& $|Y_0-Y^0|$ & $|Z_0^1-Z^0_1|$ & $|Z_0^2-Z^0_2|$ 
		& $|Y_0-Y^0|$ & $|Z_0^1-Z^0_1|$& $|Z_0^2-Z^0_2|$ \\
		\hline
		8   
		& $8.03{\times}10^{-4}$ & $7.88{\times}10^{-4}$ & $1.56{\times}10^{-3}$ 
		& $7.89{\times}10^{-4}$ & $2.28{\times}10^{-3}$ & $4.66{\times}10^{-3}$ \\
		
		16  
		& $1.98{\times}10^{-4}$ & $1.92{\times}10^{-4}$ & $4.02{\times}10^{-4}$ 
		& $1.77{\times}10^{-4}$ & $5.99{\times}10^{-4}$ & $1.21{\times}10^{-3}$ \\
		
		32  
		& $4.58{\times}10^{-5}$ & $5.39{\times}10^{-5}$ & $1.11{\times}10^{-4}$ 
		& $3.69{\times}10^{-5}$ & $1.60{\times}10^{-4}$ & $3.24{\times}10^{-4}$ \\
		
		64  
		& $9.70{\times}10^{-6}$ & $1.60{\times}10^{-5}$ & $3.24{\times}10^{-5}$ 
		& $6.33{\times}10^{-6}$ & $4.42{\times}10^{-5}$ & $8.89{\times}10^{-5}$ \\
		
		128 
		& $1.71{\times}10^{-6}$ & $5.01{\times}10^{-6}$ & $1.01{\times}10^{-5}$ 
		& $4.49{\times}10^{-7}$ & $1.27{\times}10^{-5}$ & $2.54{\times}10^{-5}$ \\
		\hline
		Order 
		& 2.210 & 1.818 & 1.818 
		& 2.637 & 1.874 & 1.881 \\
		\hline
	\end{tabular}
\end{table}
\end{example}

\section{Conclusions}

In this paper, we split the generator $f$ into the sum of two functions. In the computation of the value process $Y$, explicit and implicit schemes are alternately applied to these two generators, while the algorithms from \citep{ZhaoLi2014} are used for the control process $Z$. We rigorously prove that the two new schemes have second-order convergence rate. The proposed splitting methods show clear advantages for equations whose generator consists of a linear part and a nonlinear part, as they reduce the number of iterations required for solving implicit schemes, thereby decreasing computational cost while maintaining second-order convergence.

\renewcommand{\refname}

\end{document}